\documentclass[12pt]{amsart}
\usepackage{graphicx}
\usepackage{hyperref}
\usepackage{setspace}
\usepackage{enumerate}
\usepackage{amsmath,amsfonts,amssymb,amscd}

\newtheorem{thm}{Theorem}[section]
\newtheorem{cor}[thm]{Corollary}
\newtheorem{conj}[thm]{Conjecture}
\newtheorem{lem}[thm]{Lemma}
\newtheorem{prop}[thm]{Proposition}

\theoremstyle{definition}

\theoremstyle{remark}

\numberwithin{equation}{section}

\makeatletter
\newtheorem*{rep@theorem}{\rep@title} \newcommand{\newreptheorem}[2]{%
\newenvironment{rep#1}[1]{%
\def\rep@title{\bf #2 \ref{##1}}%
\begin{rep@theorem} }%
{\end{rep@theorem} } }
\makeatother

\newreptheorem{theorem}{Theorem}
\newreptheorem{lemma}{Lemma}
\newreptheorem{proposition}{Proposition}

\newcommand{\bigO}[1]{O \left( #1 \right)}

\begin{document}

\title[{ Sidon Basis in polynomial rings over finite fields }]
{Sidon Basis in polynomial rings over \\ finite fields}

\author{Wentang Kuo}
\address{Department of Pure Mathematics \\
University of Waterloo \\
Waterloo, ON\\  N2L 3G1 \\
Canada}
\email{wtkuo@uwaterloo.ca}

\author{Shuntaro Yamagishi}
\address{Department of Pure Mathematics \\
University of Waterloo \\
Waterloo, ON\\  N2L 3G1 \\
Canada}
\email{syamagis@uwaterloo.ca}
\indent

\date{Revised on \today}

\begin{abstract}
Let $\mathbb{F}_q[t]$ denote the ring of polynomials over $\mathbb{F}_q$, the finite field of $q$ elements.
Suppose the characteristic of $\mathbb{F}_q$ is not $2$ or $3$.
In this paper, we prove an $\mathbb{F}_q[t]$-analogue of
results 
related to the conjecture of Erd\H{o}s on the existence of infinite Sidon
sequence of positive integers which is an asymptotic basis of order 3. We prove that
there exists a $B_2[2]$ sequence of non-zero polynomials in $\mathbb{F}_q[t]$,
which is an asymptotic basis of order $3$. We also prove that
for any $\varepsilon> 0$, there exists a sequence of non-zero polynomials in $\mathbb{F}_q[t]$,
which is a Sidon basis of order $3 + \varepsilon$. In other words, there exists a
sequence of non-zero polynomials in $\mathbb{F}_q[t]$ such that any
$n \in \mathbb{F}_q[t]$ of sufficiently large degree can be expressed as a sum of four elements
of the sequence, where one of them has a degree less than or equal to $\varepsilon \deg n.$
\end{abstract}



\maketitle

\section{Introduction}
A sequence of positive integers $\omega$ is called a \textit{Sidon sequence},
if all the sums $a + a' \ (a, a' \in \omega, a \leq a')$ are distinct.
We say $\omega$ is an \textit{asymototic basis of order $g$},
if any positive integer $n$ sufficiently large can be expressed as a sum of $g$ elements of $\omega$.
If $\omega$ is a Sidon sequence and also an asymototic basis of order $g$, we say $\omega$ is
a \textit{Sidon basis of order $g$.}
An introduction to the topic is given in \cite{C}, which we paraphrase here.
It is known that there can not be a Sidon basis of order $2$. The following is
a conjecture of Erd\H{o}s \cite{E1, E2, E3}.
\begin{conj} \label{erdos}
There exists a sequence of positive integers, which is a Sidon basis
of order $3$.
\end{conj}
There had been progress made towards this conjecture. J.-M. Deshoulliers and A. Plagne \cite{DP}
constructed a Sidon basis of order $7$, and S. Kiss \cite{K1} proved the existence of a Sidon basis
of order $5$. S. Kiss, E. Rozgonyi and C. S\'{a}ndor \cite{KRS} proved that there exists a Sidon basis
of order $4$.

The focus of this paper is on two theorems toward this conjecture proved in \cite{C}.
We introduce some notations before we state these theorems.
A sequence of positive integers $\omega$ is a $B_2[g]$\textit{ sequence}
if any integer $n$ has at most $g$ representations of the form
$n = a + a' \ (a, a' \in \omega, a \leq a')$. Conjecture \ref{erdos}
can be restated as follows: There exists a $B_2[g]$ sequence with $g=1$,
which is an asymptotic basis of order $3$. The following theorem was
proved in \cite{C}.

\begin{thm}[Theorem 1.2, \cite{C}]
\label{C 1}
There exists a $B_2[2]$ sequence,
which is an asymptotic basis of order $3$.
\end{thm}

For any $\varepsilon$ such that $0 < \varepsilon < 1$, we say that $\omega$ is an
\textit{asymptotic basis of order $g + \varepsilon$} if any positive integer $n$ sufficiently
large can be represented in the following form,
$$
n = a_1 + ... + a_{g+1} \ (a_1, ..., a_{g+1} \in \omega) \  \ \text{ and } \  \min_{1 \leq i \leq g+1} a_i \leq n^{\varepsilon}.
$$
The following theorem was also proved in \cite{C}.

\begin{thm}[Theorem 1.3, \cite{C}]
\label{C 2}
For any $\varepsilon> 0$,
there exists a Sidon basis of order $3 + \varepsilon$.
\end{thm}

Let $\mathbb{F}_q$ be the finite field of $q$ elements. In this paper, we follow the approach of \cite{C} and prove
$\mathbb{F}_q[t]$-analogue of Theorems \ref{C 1} and \ref{C 2}. Let $\omega$ be a sequence of
non-zero polynomials in $\mathbb{F}_q[t]$.
We define \textit{Sidon sequence} and \textit{$B_2[g]$ sequence}
for sequences of non-zero polynomials in $\mathbb{F}_q[t]$ in a similar manner as for sequences of positive integers.
We say $\omega$ is an \textit{asymototic basis of order $g$} if any $n \in \mathbb{F}_q[t]$ with $\deg n$ sufficiently large can be expressed as a sum of $g$ elements of $\omega$. Similarly, if $\omega$ is a Sidon sequence and also an asymototic basis of order $g$, we say $\omega$ is
a \textit{Sidon basis of order $g$.}
We prove the following theorem, which is an analogue of Theorem \ref{C 1}.
\begin{thm}
\label{main 1}
Let $p$ be a prime, $p>3$, and $q = p^h \ ( h \in \mathbb{N})$.
Then there exists a $B_2[2]$ sequence of non-zero polynomials in $\mathbb{F}_q[t]$,
which is an asymptotic basis of order $3$.
\end{thm}

For any $\varepsilon$ such that $0 < \varepsilon < 1$, we say that $\omega$ is an
\textit{asymptotic basis of order $g + \varepsilon$} if any polynomial $n \in \mathbb{F}_q[t]$ with $\deg n$ sufficiently
large can be represented in the following form,
$$
n = a_1 + ... + a_{g+1} \ (a_1, ..., a_{g+1} \in \omega) \  \ \text{ and } \  \min_{1 \leq i \leq g+1} \deg a_i \leq \varepsilon \deg n.
$$
We also prove the following theorem, which is an analogue of Theorem \ref{C 2}.
\begin{thm}
\label{main 2}
Let $p$ be a prime, $p>3$, and $q = p^h \ ( h \in \mathbb{N})$.
For any $\varepsilon> 0$, there exists a sequence of non-zero polynomials in $\mathbb{F}_q[t]$,
which is a Sidon basis of order $3 + \varepsilon$.
\end{thm}

In \cite[Theorem 1.1]{C}, J. Cilleruelo proved Conjecture \ref{erdos} in the setting of $\mathbb{Z} / (M \mathbb{Z})$ for $M$ sufficiently large. For each $N \in \mathbb{N}$, let $\mathbb{G}_N = \{ f \in \mathbb{F}_q[t] : \deg f < N \}$. We also prove a $\mathbb{G}_N$-analogue of Conjecture \ref{erdos} when $N$ is a sufficiently large multiple of $4$, and the characteristic of $\mathbb{F}_q$ is not $2$ or $3$.

\begin{thm}
	\label{thm 2.1 in C}
	Let $p$ be a prime, $p>3$, and $q = p^h \ ( h \in \mathbb{N})$.
	Then for $M_0 \in \mathbb{N}$ sufficiently large, there exists a Sidon set $S = S(q, M_0)$ in $\mathbb{G}_{4M_0}  \subseteq \mathbb{F}_q[t]$
	such that the following holds.
	Given any $g \in \mathbb{G}_{4M_0}$, there exist
	$s_1, s_2, s_3 \in S$ with $s_i \not = s_j \ (i \not = j)$ such that
	$$
	s_1 + s_2 + s_3 = g.
	$$
\end{thm}

The organization of the paper is as follows.
In Section \ref{second sec}, we prove Theorem \ref{thm 2.1 in C} and its corollary, which become useful in the proof of Theorems \ref{main 1} and \ref{main 2}.
We introduce notation and results from probabilistic methods in Section \ref{prelim}.
We then prove our main results Theorems \ref{main 1} and \ref{main 2} in Sections \ref{sec proof of main 1} and \ref{sec proof of main 2},
respectively. We obtain several bounds in Section \ref{appendix}. These bounds are used in Sections \ref{calculations 1} and \ref{calculations 2}, where we present calculations of estimates used in Sections \ref{sec proof of main 1} and \ref{sec proof of main 2}, respectively. We note that this paper
involves a considerable amount of computation, some of which is similar to that of \cite{C}.
In an effort to keep the paper concise,
we omitted the details of some calculations, most notably in
Sections \ref{sec proof of main 2} and \ref{calculations 2} as they are similar to that of Sections \ref{sec proof of main 1} and \ref{calculations 1},
respectively. Also, we assume that the characteristic of $\mathbb{F}_q$ is not $2$ or $3$ for the remainder of the paper, unless
it is explicitly stated otherwise.



\section{Proof of Conjecture \ref{erdos} for $\mathbb{G}_{N}$}
\label{second sec}

Let $G$ be an abelian group. For any subset $A \subseteq G$ and $x \in G$,
we denote $r_{A-A}(x)$ to be the number of representations of the form $x = a - a' \  (a, a' \in A)$.
We say that a set $A \subseteq G$ is a \textit{Sidon set} if $r_{A-A}(x) \leq 1$
whenever $x \not = 0$. This condition is equivalent to saying that the representation of
elements of $G$ as a sum of two elements of $A$ is unique if it exists. In other words,
if for some $a,b,c,d \in A$ we have $a+b = c+d$, then either we have
$a=c$, $b=d$ or $a=d$, $b=c$.
Recall from above that we defined  $\mathbb{G}_N = \{ f \in \mathbb{F}_q[t] : \deg f < N \}$, which is a group under addition.
In this section, we prove Theorem \ref{thm 2.1 in C} and Corollary \ref{cor to thm}, which
are $\mathbb{F}_q[t]$-analogue of \cite[Theorem 2.1]{C} and \cite[Corollary 2.1]{C}, respectively.

We recall the statement of Theorem \ref{thm 2.1 in C}, which confirms Conjecture \ref{erdos} for $\mathbb{G}_N$
when $N$ is a sufficiently large multiple of $4$, and the characteristic of $\mathbb{F}_q$ is not $2$ or $3$.
\begin{reptheorem}{thm 2.1 in C}
Let $p$ be a prime, $p>3$, and $q = p^h \ ( h \in \mathbb{N})$.
Then for $M_0 \in \mathbb{N}$ sufficiently large, there exists a Sidon set $S = S(q, M_0)$ in $\mathbb{G}_{4M_0}  \subseteq \mathbb{F}_q[t]$
such that the following holds.
Given any $g \in \mathbb{G}_{4M_0}$, there exist
$s_1, s_2, s_3 \in S$ with $s_i \not = s_j \ (i \not = j)$ such that
$$
s_1 + s_2 + s_3 = g.
$$
\end{reptheorem}
\begin{proof}
We have the following group isomorphisms when we only consider the additive properties,
$$
\mathbb{G}_{4M_0} \cong (\mathbb{F}_q)^{4M_0} \cong (\mathbb{Z} / p \mathbb{Z})^{4hM_0}
\cong \mathbb{F}_{q'} \times \mathbb{F}_{q'},
$$
where $q' = p^{2hM_0}$. Therefore, if we can find a Sidon basis of order $3$ in
$\mathbb{F}_{q'} \times \mathbb{F}_{q'}$, then we are done.

Let $S = \{(x,x^2) : x \in \mathbb{F}_{q'} \}$. Then, by \cite{C1} we know that $S$ is a Sidon set in $\mathbb{F}_{q'} \times \mathbb{F}_{q'}$.
For the sake of completeness, we present the proof from \cite{C1} here.
We have to check that given $(0,0) \not = (e_1, e_2) \in \mathbb{F}_{q'} \times \mathbb{F}_{q'}$, the equation $(x_1, x_1^2) - (x_2, x_2^2) = (e_1, e_2)$
uniquely determines $x_1$ and $x_2$ in $\mathbb{F}_{q'}$, or that it has no solution.
If $e_1 =0$, then it is clear that there do not exist $x_1$ and $x_2$ in  $\mathbb{F}_{q'}$ that satisfy the equation.
On the other hand, suppose $e_1 \not = 0$.
Since $x_1 = e_1 + x_2$, we have
$e_2 = (x_2 + e_1)^2 - x_2^2 =  2 e_1 x_2 + e_1^2$, which uniquely determines
$x_2$. Once $x_2$ is determined, there is only one choice for $x_1$.
Therefore, we have shown that $r_{S-S}( (e_1, e_2)  ) \leq 1$, and hence $S$ is a Sidon set.

Now we show $S$ is an additive basis of order $3$. This is equivalent to showing that
for any $(a,b) \in \mathbb{F}_{q'} \times \mathbb{F}_{q'}$,
the system
\begin{equation}
\label{system}
x+y+t = a \ \text{  and  } \ x^2 +y^2+ t^2 = b,
\end{equation}
has a solution in $\mathbb{F}_{q'} \times \mathbb{F}_{q'} \times \mathbb{F}_{q'}$.

We consider the polynomial
$$
f(x,y) = x^2 +y^2+ (x+y-a)^2 - b = 2 (x^2 + y^2 + xy - a x - a y ) + a^2 - b
$$
constructed from ~(\ref{system}), and its homogenization
$$
F(x,y,z) = 2 (x^2 + y^2 + x y - a x z - a y z ) + (a^2 - b)z^2.
$$

Suppose $F$ is reducible over $\overline{ \mathbb{F} }_{q'}$, where $\overline{\mathbb{F} }_{q'}$
is the algebraic closure of $\mathbb{F}_{q'}$, in which case
$F$ decomposes into two lines $L_1$ and $L_2$
with coefficients in $\overline{\mathbb{F} }_{q'}$.
Without loss of generality, let
$$
F(x,y,z) = 2(x + \alpha_1  y + \beta_1 z)(x + \alpha_2 y + \beta_2 z),
$$
where $\alpha_1, \beta_1, \alpha_2, \beta_2 \in \overline{\mathbb{F} }_{q'}$.
By expanding out the factors, we see from the coefficients of
$y^2, xy, xz$, and $yz$
that
$\alpha_1\alpha_2 =1, \alpha_1 + \alpha_2 = 1$,
$\beta_1 + \beta_2 = -a$, and $\alpha_1\beta_2 + \alpha_2\beta_1 = -a$, respectively.
Since $q' = p^{2hM_0}$ and $2 | (2hM_0)$, we have $\mathbb{F}_{p^2} \subseteq \mathbb{F}_{q'}.$
From the first and the second equation, we obtain that $\alpha_1$ and $\alpha_2$ are non-zero, and
$$
\alpha_1, \alpha_2 \in \mathbb{F}_{p^2} \subseteq \mathbb{F}_{q'}.
$$
Since the characteristic of $\mathbb{F}_{q'}$ is not $3$, we also obtain $\alpha_1 \not = \alpha_2.$
Then from the third and the forth equation, we can deduce that
$$
\beta_1, \beta_2 \in \mathbb{F}_{q'}.
$$
Therefore, $F$ is in fact reducible over $\mathbb{F}_{q'}$, and hence
$f$ decomposes into two linear factors over $\mathbb{F}_{q'}$ as follows
$$
f(x,y) = F(x,y,1) = 2(x + \alpha_1  y + \beta_1)(x + \alpha_2 y + \beta_2).
$$
Thus, we see that
~(\ref{system}) has at least $q'$ solutions in $\mathbb{F}_{q'} \times \mathbb{F}_{q'} \times \mathbb{F}_{q'}$
in this case.

On the other hand, suppose $F$ is irreducible over $\overline{\mathbb{F}}_{q'}$. 
Let $V(F)$ be the hypersurface in $\mathbb{P}^2_{ \mathbb{F}_{q'}}$ defined by $F$.
In this case, we may invoke a theorem by S. Lang and A. Weil \cite{LW}, and obtain that $V(F)$ has $q' + O(1)$ rational points
over $\mathbb{F}_{q'}$. We know that $F(x,y,0) = 2 (x^2 + y^2 + x y)$
decomposes into two linear factors over $\overline{\mathbb{F} }_{q'}$, because it is a quadratic form in two variables. Then we can verify that
$F(x,y,0)$ has at most $O(1)$ solutions in $\mathbb{P}^1_{ \mathbb{F}_{q'}}$. Therefore,
it follows that $V(F)$ contains $q' + O(1)$ points
of the form $[x_0:y_0:1]$ from which we deduce
~(\ref{system}) has $q' + O(1)$ solutions in $\mathbb{F}_{q'} \times \mathbb{F}_{q'} \times \mathbb{F}_{q'}$.

In both cases, we have that ~(\ref{system})
has at least $q' + O(1)$ solutions. Suppose $(x_1, x_2, x_3)$ is a solution to ~(\ref{system})
such that $x_i = x_j$ for some $i \not = j$, without loss of generality let $i=1$ and $j=2$. Then, the number of such solutions is equal to the
number of solutions to
\begin{equation}
\label{system 2}
x+x+y = a \ \text{  and  } \ x^2 +x^2+ y^2 = b.
\end{equation}
Since the equation $2 x^2 + (a - 2x)^2 = b$ has at most
$2$ solutions in $\mathbb{F}_{q'}$, we have that
~(\ref{system 2}) has at most $2$ solutions.
Hence, the number of solutions $(x_1, x_2, x_3)$ to ~(\ref{system})
such that $x_i = x_j$ for some $i \not = j$ is $O(1)$.
Therefore, for each $(a,b) \in \mathbb{F}_{q'} \times \mathbb{F}_{q'}$ we can find a solution $(x_1, x_2, x_3)$ to ~(\ref{system})
such that $x_i \not = x_j \ (i \not = j)$, provided $q'$ is sufficiently large.
\end{proof}

\begin{cor}
\label{cor to thm}
Let $p$ be an odd prime, and $q = p^h \ ( h \in \mathbb{N})$.
Then for $M_0 \in \mathbb{N}$ sufficiently large, there exists a Sidon set $S = S(q, M_0)$ in $\mathbb{G}_{4M_0}\subseteq \mathbb{F}_q[t]$
such that the following holds.
Given any $g \in \mathbb{G}_{4M_0}$, there exist
$s_1, s_2, s_3, s_4 \in S$ with $s_i \not = s_j \ (i \not = j)$ such that
$$
s_1 + s_2 + s_3 + s_4 = g.
$$
\end{cor}
\begin{proof}
Let $\mathbb{F}_{q'}$ and $S \subseteq \mathbb{F}_{q'} \times \mathbb{F}_{q'}$ be as in the proof of Theorem \ref{thm 2.1 in C}.
We show that $S$ satisfies the required conditions. From the proof of Theorem
\ref{thm 2.1 in C}, we know that
for any $(a,b) \in \mathbb{F}_{q'} \times \mathbb{F}_{q'}$, the system
\begin{equation}
\label{system 3}
x + y + t + 0 = a \ \text{  and  } \ x^2 +y^2+ t^2 + 0^2= b,
\end{equation}
has at least $q' + \bigO{1}$ solutions
of the form $(x_1, x_2, x_3) \in \mathbb{F}_{q'} \times \mathbb{F}_{q'} \times \mathbb{F}_{q'}$, where
$x_i \not = x_j \ (i \not = j)$.
We observe that all of these solutions except those with at least one of $x_1, x_2, x_3$ being $0$ satisfy the conditions.
Without loss of generality, suppose $x_3=0$, then ~(\ref{system 3}) reduces to solving
\begin{equation}
x + y = a \ \text{  and  } \ x^2 +y^2 = b,
\end{equation}
which further reduces to solving a quadratic equation.
Thus, it follows that the number of solutions $(x_1, x_2, x_3)$ to ~(\ref{system 3})
with at least one of $x_1, x_2, x_3$ being $0$ is $\bigO{1}$.
Therefore, it follows that there exist at least $q'+ \bigO{1}$ solutions in $\mathbb{F}_{q'} \times \mathbb{F}_{q'} \times \mathbb{F}_{q'}$,
which satisfy the desired conditions.
\end{proof}

\section{Preliminaries}
\label{prelim}
We begin this section by introducing a result that is useful to us.
The following result is known as the Borel-Cantelli lemma, which plays a crucial role
in probability theory \cite{HR}.
\begin{thm}[The Borel-Cantelli lemma]
\label{Borel-C}
Suppose we have a probability space $(\Omega, \mathcal{M}, \mathbb{P})$.
Let $\{ E_j \}_{j \geq 1}$ be a sequence of measurable sets.
If
$$
\sum_{j=1}^{\infty} \mathbb{P}(E_j) < \infty,
$$
then we have
$$
\mathbb{P} \left(  \cap_{i=1}^{\infty}  \cup_{j=i}^{\infty} E_j  \right) = 0.
$$
\end{thm}
In other words, the Borel-Cantelli lemma states that if
$\sum_{j=1}^{\infty} \mathbb{P}(E_j) < \infty$, then with probability $1$ at most a finite
number of the events $E_j$ can occur.

Throughout the paper we fix $N$ to be a sufficiently large
positive integer, and we let $S$ be a non-empty subset of $\mathbb{G}_N$. Furthermore, we
choose $S$ to satisfy the conditions of Theorem \ref{thm 2.1 in C} in Section \ref{sec proof of main 1},
and we
choose $S$ to satisfy the conditions of Corollary \ref{cor to thm} in Section \ref{sec proof of main 2}.
Let $\Omega$ be the space of all sequences of polynomials in $\mathbb{F}_q[t]$.
Let
$n_0 = t^N \in \mathbb{F}_q[t]$, and by $x \equiv S (\text{mod }n_0)$, we mean $x \equiv s  (\text{mod }n_0)$
for some $s \in S$.
For each $\gamma < 1$ and $M \in \mathbb{N}$, we define the probability space $\mathcal{S}_M( \gamma ; S \ \text{mod } n_0 )$ in the following manner. We let $\mathcal{S}_M( \gamma ; S \ \text{mod } n_0 ) = (\Omega, \mathcal{M}, \mathbb{P})$ to be the probability space
of all sequences of polynomials $\omega$, where $\mathcal{M}$ is the appropriate $\sigma$-algebra, such that all the events $x \in \omega$ are independent, and
\begin{eqnarray}
\mathbb{P}(  \{ x \in \omega \})
=
\begin{cases}
q^{- \gamma (\deg x) }, &\mbox{if } x \equiv S (\text{mod }n_0) \text{ and } \deg x > M, \\
0, & \mbox{otherwise. }
\end{cases}
\end{eqnarray}
We refer the reader to \cite{HR} for the details on construction of such probability spaces.
For simplicity, we let $\mathbb{P}(  \{  0 \in \omega \} ) = 0$.
From here on whenever we refer to $\mathbb{P}$ we mean this probability measure.
Let $f$ be a function from $\mathbb{R}$ to $\mathbb{R}$.
By $f = o_M(1)$, we mean that $|f(M)| \rightarrow 0$ as $M \rightarrow \infty.$

The following result is known as Janson's inequality, see for example
\cite{C, J, JLR}.
\begin{thm}[Janson's inequality]
\label{Janson}
Let $\mathcal{F}$ be a family of sets, and let $\omega$ be a random subset.
Let $Y(\omega) = | \{ \theta \in \mathcal{F} : \theta \subseteq \omega \} |$
with finite expected value $\mu = \mathbb{E}(Y(\omega))$. Then, for $0 \leq \varepsilon \leq 1$, we have that
$$
\mathbb{P}( \{ \omega \in \Omega : Y(\omega) \leq (1 - \varepsilon ) \mu \} ) \leq \exp( - \varepsilon^2 \mu^2/ (2 \mu  + 2  \Delta(\mathcal{F})   ) )  ,
$$
where
\begin{equation}
\label{def Delta}
\Delta(\mathcal{F}) = \sum_{\stackrel{\theta, \theta' \in \mathcal{F} }{\theta \sim \theta'}} \mathbb{P}( \{ \omega \in \Omega:  \theta, \theta' \subseteq \omega \}),
\end{equation}
and $\theta \sim \theta'$ means $\theta \cap \theta' \not = \emptyset$
and $\theta \not = \theta'$. In particular, if $\Delta(\mathcal{F})< \mu$ we have
$$
\mathbb{P}(  \{ \omega \in \Omega : Y(\omega) \leq \mu/2 \} ) \leq \exp( - \mu/ 16 ).
$$
\end{thm}
Note in order to avoid clutter in the exposition, whenever we have a subset of $\Omega$ of the form
$\{ \omega \in \Omega :  \omega \text{ satisfies  ... }   \}$ we simply denote it by
$\{  \omega \text{ satisfies  ... }  \}$.

For a given vector $\overline{y} = (y_1, ..., y_H)$, we define
Set$(\overline{y}) = \{y_1, ..., y_H\}$. We say that a collection of $K$ distinct
vectors $\overline{x}_j \ (1 \leq  j \leq K) $ form a \textit{disjoint set of $K$
vectors} ($K$-d.s.v. for short) if
Set$(\overline{x}_j) \bigcap$ Set$(\overline{x}_{l}) = \emptyset$
for any $j \not = l, 1 \leq j, l \leq K$.
We say that $K$ distinct vectors with $H$ coordinates form a \textit{vectorial sunflower
of $K$ petals}, if for some $I \subseteq \{ 1, ..., H \}$ the following two
conditions are satisfied:

i) For all $i \in I$, all the vectors have the same $i$-th coordinate.

ii) The set of vectors obtained by removing the $i$-th coordinates, for
all $i \in I$, form a $K$-d.s.v.

Following the terminology of \cite{C}, we say $\mathcal{F}$ is a family of vectors of $H$ coordinates if $\mathcal{F}$ is a subset of $(\mathbb{F}_q[t])^H$. We have the following lemma.
\begin{lem}[Vectorial sunflower lemma]
\label{vs lem}
Let $\mathcal{F}$ be a family of vectors of $H$ coordinates.
If $\mathcal{F}$ does not contain a vectorial sunflower of $K$ petals,
then
$$
|\mathcal{F}| \leq H! ((H^2 - H +1) K)^H.
$$
\end{lem}
\begin{proof}
This is obtained by slightly modifying the proof of \cite[Lemma 3.2]{C}.
\end{proof}

Given $\mathcal{F}$, a family of vectors of $H$ coordinates,
we define
$$
\mathcal{F}(\omega) = \{ \overline{x} \in \mathcal{F} : \text{Set}(\overline{x}) \subseteq \omega \}.
$$
The following is an immediate consequence of Lemma \ref{vs lem}.

\begin{cor}
\label{cor sun and bound}
Let $\{ \mathcal{F}_n \}_{n \in \mathbb{F}_q[t]}$ be a sequence of family of vectors of $H$ coordinates.
Suppose for $\Omega(K) = \{ \omega \in \Omega : \mathcal{F}_n(\omega) \text{ does not contain vectorial sunflowers of }
K \text{ petals for any } n \in \mathbb{F}_q[t] \}$,
we have
$$
\mathbb{P}(\Omega(K) ) = 1 - o_M(1).
$$
Then, we have
$$
\mathbb{P}(\{  |\mathcal{F}_n(\omega)| \leq H! ((H^2 - H +1) K)^H \text{ for all } n \in \mathbb{F}_q[t] \}) \geq 1 - o_M(1).
$$
\end{cor}

We also make use of the following proposition.
\begin{prop}
\label{K-dsv 1}
Let $\{ \mathcal{F}_n \}_{n \in \mathbb{F}_q[t]}$ be a sequence of family of vectors of $H$ coordinates, and $\{ \mathcal{F}_n(\omega) \}_{n \in \mathbb{F}_q[t]}$ the corresponding random family, where $\omega$ is a random sequence in $\mathcal{S}_M( \gamma ; S \ \text{mod } n_0 )$.
Suppose
there is $\delta > 0$ such that $\mathbb{E} (|\mathcal{F}_n(\omega)|) \ll q^{- \delta \max\{ \deg n, M \} }$ for all $n \in \mathbb{F}_q[t]$.
If $K > 1 / \delta$, then
$$
\mathbb{P}( \{  \mathcal{F}_n(\omega) \text{ contains a $K$-d.s.v. for some } n \in \mathbb{F}_q[t] 
\} ) = o_M(1).
$$
\end{prop}

\begin{proof}
By unraveling the definitions, we have the following sequence of inequalities
\begin{eqnarray}
&&\mathbb{P}( \{  \mathcal{F}_n(\omega) \text{ contains a $K$-d.s.v.}\} )
\notag
\\
&\ll&
\sum_{ \stackrel{\overline{x}_1, ..., \overline{x}_K \in \mathcal{F}_n  }{ \text{form a $K$-d.s.v. } } }
\mathbb{P}( \{  \text{Set }(\overline{x}_1), ..., \text{Set }(\overline{x}_K) \subseteq \omega \})
\notag
\\
&=&
\sum_{ \stackrel{\overline{x}_1, ..., \overline{x}_K \in \mathcal{F}_n  }{ \text{form a $K$-d.s.v. } } }
\mathbb{P}( \{  \text{Set }(\overline{x}_1) \subseteq \omega \}) ... \ \mathbb{P}( \{  \text{Set }(\overline{x}_K) \subseteq \omega \})
\notag
\\
&\leq&
\frac{1}{K!} \left( \sum_{ \overline{x} \in \mathcal{F}_n  } \mathbb{P}( \{  \text{Set }(\overline{x}) \subseteq \omega \} ) \right)^K
\notag
\\
&=&
\frac{ \mathbb{E} (|\mathcal{F}_n(\omega)|)^K }{K!}
\notag
\\
&\ll&
\frac{ q^{-\delta K \max\{ \deg n, M \}} }{K!}.
\notag
\end{eqnarray}
Since $(1 - \delta K)< 0$, we obtain that
\begin{eqnarray}
&&\mathbb{P}( \{  \mathcal{F}_n(\omega) \text{ contains a $K$- d.s.v. for some } n \in \mathbb{F}_q[t] \} 
)
\notag
\\
&\ll&
\sum_{ \deg n \leq M } \mathbb{P}( \{ \mathcal{F}_n(\omega) \text{ contains a $K$- d.s.v.}\} ) + \sum_{ \deg n > M } \mathbb{P}( \{ \mathcal{F}_n(\omega) \text{ contains a $K$- d.s.v.}\} )
\notag
\\
&\ll&
q^{-\delta K \max\{ \deg 0, M \}} + \sum_{j \leq M} (q^{j+1} - q^j) {q^{-\delta K M}} + \sum_{j > M} (q^{j+1} - q^j) {q^{-\delta K j}}
\notag
\\
&\ll&
q^{-\delta K M} + {q^{-\delta K M}} \sum_{j \leq M} q^j + \sum_{j > M} {q^{(1- \delta K)j}}
\notag
\\
&=& O(q^{(1 - \delta K) M})
\notag
\\
&=& o_M(1).
\notag
\end{eqnarray}
\end{proof}

\section{Proof of Theorem \ref{main 1} }
\label{sec proof of main 1}

In this section, we consider the probability space
$\mathcal{S}_M( \gamma ; S \ \text{mod } n_0 )$, where
we let $\gamma = \frac{7}{11}$, and let
$S$ to be a non-empty subset of $\mathbb{G}_N$
satisfying the conditions of Theorem \ref{thm 2.1 in C}.
The basic strategy is as follows.
We use the Borell-Cantelli Lemma (Theorem \ref{Borel-C}) to show that in the probability space $\mathcal{S}_M( \gamma ; S \ \text{mod } n_0 )$,
``most'' of the sequences, in other words with probability $1$, has ``many'' representations of $n$ as a sum of three of its elements
for all $n \in \mathbb{F}_q[t]$ with $\deg n $ sufficiently large. We then show that out of these sequences, there exists
a sequence such that even after removing some of its elements to make it $B_2[2]$, it still has at least one
representation of $n$ as a sum of three of its elements for each $n \in \mathbb{F}_q[t]$ with $\deg n $ sufficiently large.

For each $n \in \mathbb{F}_q[t]$, we consider the following collection of sets
$$
\mathcal{Q}_n = \{ \theta = \{ x_1, x_2, x_3 \} \subseteq \mathbb{F}_q[t] : x_1 + x_2 + x_3 = n, \ x_i \not \equiv x_j (\text{mod }n_0) \text{ for } \ i \not = j  \}.
$$
Given a sequence of polynomials $\omega$, we let
$$
\mathcal{Q}_n(\omega) = \{ \theta \in \mathcal{Q}_n : \theta \subseteq \omega \}.
$$

We define
$$
\mathcal{T}_n = \{ \overline{x} = (x_1, x_2, x_3, x_4, x_5, x_6, x_7,x_8) : \overline{x} \text{ satisfies Cond} (\mathcal{T}_n)  \},
$$
where
\begin{eqnarray}
\phantom{1122}
\text{Cond} (\mathcal{T}_n)
=
\begin{cases}
\{ x_1, x_2, x_3 \} \in \mathcal{Q}_n, \\
x_1 + x_4 = x_5 + x_6 = x_7 + x_8, \ \ \  \{x_1, x_4\} \not =  \{x_5, x_6 \} \not =  \{x_7, x_8 \},\\
x_1 \equiv x_5 \equiv x_7 \ (\text{mod }n_0), \ x_4 \equiv x_6\equiv x_8 \ (\text{mod }n_0).
\end{cases}
\end{eqnarray}
We also let
$$
\mathcal{T}_n(\omega) = \{ \overline{x} \in \mathcal{T}_n : \text{Set}(\overline{x}) \subseteq \omega \}.
$$

The \textit{$B_2[2]$-lifting process} of a sequence $\omega$ consists of removing from
$\omega$ those elements $a_1 \in \omega$ such that there exist $a_2, a_3, a_4, a_5, a_6 \in \omega$
with $a_1 + a_2 = a_3 + a_4 =  a_5 + a_6 $ and $\{a_1, a_2\} \not =  \{a_3, a_4 \} \not =  \{a_5, a_6 \}$.
We denote by $\omega_{B_2[2]}$ the resulting $B_2[2]$ sequence obtained by applying this process to $\omega$.

The quantity $|\mathcal{T}_n(\omega)|$ provides an upper bound for the number of representations of $n$ counted in $\mathcal{Q}_n(\omega)$
that are destroyed in the $B_2[2]$-lifting process of $\omega$ for the following reason.
Suppose that $\theta = \{x_1, x_2, x_3 \} \in \mathcal{Q}_n(\omega)$ contains an element,
say $x_1$, which is removed in the $B_2[2]$-lifting process.
Then, there exist $x_4, x_5, x_6, x_7, x_8 \in \omega$, which satisfy
$x_1 + x_4 = x_5 + x_6 = x_7 + x_8$ and $\{ x_1, x_4 \} \not = \{ x_5, x_6 \} \not =  \{x_7, x_8 \}$.
Since all $x_i \equiv S (\text{mod }n_0)$ and
$S$ is a Sidon set in $\mathbb{G}_N$, interchanging $x_5$ with $x_6$, and $x_7$ with $x_8$
if necessary, we have $x_1 \equiv x_4 \equiv x_7 \ (\text{mod }n_0)$
and $x_5 \equiv x_6 \equiv x_8 \ (\text{mod }n_0)$.
Thus, we have a map from the set of $\theta \in \mathcal{Q}_n(\omega)$ destroyed in
the $B_2[2]$-lifting process to $\mathcal{T}_n(\omega)$, and it is easy to see that this map is injective.
Consequently, we have
\begin{equation}
\label{ineq 1}
|\mathcal{Q}_n( \omega_{B_2[2]} )| \geq |\mathcal{Q}_n( \omega )| - |\mathcal{T}_n( \omega )|.
\end{equation}
Therefore, Theorem \ref{main 1} is established if we can prove that there exists a sequence $\omega_0$
such that
for any $n \in \mathbb{F}_q[t] \backslash \{ 0 \}$ with sufficiently
large degree, we have $| \mathcal{Q}_n( \omega_0 )| \gg q^{\delta \deg n} $ for some $\delta > 0$,
and $| \mathcal{T}_n(\omega_0)| \ll 1$.
We show that in some sense there are many sequences satisfying the former
condition, and then we show it is also the case for the latter condition.
These tasks are accomplished in Propositions
\ref{lower bound Q} and \ref{upper bound T_n}.
We then prove that there exist sequences
satisfying both conditions.
Before we get into the proofs of these propositions,
we list three useful estimates. However, we postpone their proofs to Section \ref{calculations 1}.
\begin{lem}
\label{E Q_n}
We have that
$$
\mathbb{E}(| \mathcal{Q}_n(\omega) |) \gg q^{(1/11) \deg n},
$$
for $n \in \mathbb{F}_q[t] \backslash \{ 0 \}$ with $\deg n$ sufficiently large.
\end{lem}

Recall from ~(\ref{def Delta}), the definition of $\Delta (\cdot)$.
\begin{prop}
\label{Delta Q_n}
We have that
$$
\Delta( \mathcal{Q}_n) \ll q^{-\frac{2}{11} \deg n},
$$
for $n \in \mathbb{F}_q[t] \backslash \{ 0 \}$ with $\deg n$ sufficiently large.
\end{prop}

\begin{lem}
\label{exp Tn}
We have that
$$
\mathbb{E}(| \mathcal{T}_n(\omega)|) \ll q^{-\frac{1}{11} \max \{ \deg n, M \} }.
$$
\end{lem}

We now prove the following proposition.
\begin{prop}
\label{lower bound Q}
We have that
$$
\mathbb{P} ( \{ |\mathcal{Q}_n(\omega)| \gg q^{\frac{1}{11} \deg n} \} ) =1,
$$
for $n \in \mathbb{F}_q[t] \backslash \{ 0 \}$ with $\deg n$ sufficiently large.
\end{prop}
\begin{proof}
We apply Theorem \ref{Janson} with $\mathcal{F} = \mathcal{Q}_n$ and
$Y(\omega) = |\mathcal{Q}_n(\omega)| = |\{ \theta \in \mathcal{Q}_n : \theta \subseteq \omega \}|$,
where $\omega$ is a random sequence in $\mathcal{S}_M( 7/11 ; S \ \text{mod } n_0 )$.
We proved that $\mu = \mu_n = \mathbb{E}(\mathcal{Q}_n( \omega )) \gg q^{\frac{1}{11}\deg n}$ in Lemma \ref{E Q_n},
and $\Delta( \mathcal{Q}_n) \ll q^{-\frac{2}{11}\deg n}$ in Proposition \ref{Delta Q_n}.
Hence for $\deg n$ sufficiently large, we have $\Delta( \mathcal{Q}_n) < \mu_n$. Then, Theorem \ref{Janson} implies that
$$
\mathbb{P}( \{ |\mathcal{Q}_n(\omega)| \leq \mu_n/2 \}) \leq \exp( - \mu_n/ 16 ).
$$
Therefore, we obtain that for some $C, C'>0$ and $T \in \mathbb{N}$, we can write
\begin{eqnarray}
\sum_{n \in \mathbb{F}_q[t]} \mathbb{P}( \{ | \mathcal{Q}_n(\omega)| \leq \mu_n/2 \} )
&<& C' + \sum_{j = T}^{\infty} (q^{j+1}-q^j) \exp(- C q^{\frac{1}{11} j})
\notag
\\
&<& C' + (q - 1) \sum_{j = T}^{\infty} q^j \exp(- C q^{\frac{1}{11} j})
\notag
\\
&<& \infty.
\notag
\end{eqnarray}
Thus, Theorem \ref{Borel-C} implies that with probability $1$, we have
$|\mathcal{Q}_n(\omega)| > \mu_n/2 \gg q^{\frac{1}{11} \deg n}$ for all $n \in \mathbb{F}_q[t] \backslash \{ 0 \}$ with $\deg n$ sufficiently large.
\end{proof}

For each $r \in \mathbb{F}_q[t]$, we define the following families of vectors, whose expected values are
bounded in Lemma \ref{family bound}:
\begin{eqnarray}
\label{family set}
\phantom{12345} \mathcal{U}_{r} &=& \{ \overline{x} = (x_1, x_2) : x_1 + x_2 = r, x_1 \not = x_2 \},
\\
\mathcal{V}_{r} &=& \{ \overline{x} = (x_1, x_2) : x_1 - x_2 = r, x_1 \not = x_2 \},
\notag
\\
\mathcal{W}_r &=& \{ \overline{x} = (x_4, x_5, x_6, x_7, x_8) : x_5 + x_6 - x_4 = x_7 + x_8 - x_4 = r, x_i \not = x_j \ (i \not =j) \}.
\notag
\end{eqnarray}

We prove the following lemma in Section \ref{calculations 1}.
\begin{lem}
\label{family bound}
We have the following bounds on the expectations.
\newline
\newline
i) $\mathbb{E}(|  \mathcal{U}_{r}(\omega)|) \ll q^{-\frac{3}{11} \max \{ \deg r, M \} }$.
\newline
ii) $\mathbb{E}(| \mathcal{V}_{r}(\omega)|) \ll q^{-\frac{3}{11} \max \{ \deg r, M \} }$.
\newline
iii) $\mathbb{E}(| \mathcal{W}_{r}(\omega)|) \ll q^{ - \frac{2}{11} \max \{ \deg r, M \} }$. 
\end{lem}

\begin{lem}
\label{family}
Let $\mathcal{F}_r$ be any of the three families in ~(\ref{family set}), then we have
$$
\mathbb{P}( \{ \mathcal{F}_r(\omega) \text{ contains a $12$-d.s.v. for some } r \in \mathbb{F}_q[t] \})
= o_M(1).
$$
\end{lem}
\begin{proof}
For any of the three choices of $\mathcal{F}_r$, Lemma \ref{family bound} shows that 
$\mathbb{E}(|\mathcal{F}_r(\omega)|) \ll q^{-\frac{2}{11} \max\{ \deg r, M \} }$. Thus, the result
follows by Proposition \ref{K-dsv 1}.
\end{proof}

We have the following proposition, which is one of the main ingredients to
prove Theorem \ref{main 1}.
\begin{prop}
\label{upper bound T_n}
We have that
$$
\mathbb{P}(\{| \mathcal{T}_n(\omega)| \leq 10^{28} \text{ for all } n \in \mathbb{F}_q[t] \}) \geq 1 - o_M(1).
$$
\end{prop}

\begin{proof}
We claim the following statement:

\textbf{Claim.} With probability $1 - o_M(1)$, $\mathcal{T}_n(\omega)$ does not contain vectorial sunflowers of $12$ petals for
any $n \in \mathbb{F}_q[t]$.

Assuming the claim holds, we can apply Corollary \ref{cor sun and bound} to obtain
$$
\mathbb{P}(\{ |\mathcal{T}_n(\omega)| \leq 8! ((8^2 - 8 +1) 12)^8 \text{ for all } n \in \mathbb{F}_q[t] \}) \geq 1 - o_M(1).
$$
Thus, we see that proving the above claim is sufficient to obtain our result.
We prove it for distinct possible types $I \subseteq \{ 1, 2, 3, 4, 5, 6, 7, 8 \}$
of vectorial sunflowers in $\mathcal{T}_n(\omega)$. We consider various cases in a similar manner as in \cite[Proposition 4.2]{C}.

We let Case 1 be when $I = \emptyset$.
If two of the entries of the equation $x_1 + x_2 + x_3 =n$ is chosen, then the third
is uniquely determined.
Thus there can not be a vectorial sunflower of type $I$, where  $| I \cap \{ 1, 2, 3 \} | = 2$.
Let Case 2 be when $| I \cap \{ 1, 2, 3 \} | = 1$.

Let us assume $| I \cap \{ 1, 2, 3 \} | = 0$ or $3$, for otherwise
it is taken care of in Case 2.
We split into further cases. Suppose $I$ contains at least one of the
pairs $\{1,4\}$, $\{5,6\}$ or $\{7,8\}$. Then we can deduce that
$|I \cap \{ 1,4, 5, 6, 7, 8 \}| = 2, 4$ or $6$, because of the equation
$x_1 + x_4 = x_5 + x_6 = x_7 + x_8$. For example, if $\{1,4\} \subseteq I$ and $5 \in I$,
then this forces $6 \in I$. Suppose $|I \cap \{ 1, 4, 5, 6, 7, 8 \}| = 6$. Since $1 \in I$,
we have $| I \cap \{ 1, 2, 3 \} | = 3$, and hence, $I = \{ 1, 2, 3, 4, 5, 6, 7, 8 \}$.
Thus, we consider the following cases,
Case 3 when $|I \cap \{ 1, 4, 5, 6 \}| = 2$
or  $|I \cap \{ 5, 6, 7, 8 \}| = 2$ or  $|I \cap \{ 1, 4, 7, 8 \}| = 2$,
and Case 4 when $I = \{ 1, 2, 3, 4, 5, 6, 7, 8 \}$.
We see that the possibilities considered in this paragraph are all contained in either Case 3 or Case 4.

Suppose $I$ does not contain any of the pairs $\{1,4\}$, $\{5,6\}$ or $\{7,8\}$.
In this case, we have $| I \cap \{ 1,4,5,6,7,8 \} | \in \{ 0, 1, 2, 3 \}$.
If $| I \cap \{ 1,4,5,6,7,8 \} | = 0$, then $1 \not \in I$ and hence $| I \cap \{ 1, 2, 3 \} | = 0$.
Therefore, $I = \emptyset$ and this is taken care of in Case 1.
We let Case 5 be when $| I \cap \{ 1,4,5,6,7,8 \} | = 1.$
If $| I \cap \{ 1,4,5,6,7,8 \} | = 2$ or $3$, then it is easy to see
that these possibilities are taken care of in Case 3.
Therefore, it is sufficient to only consider the above five cases of distinct types of $I$.

Case 1. $I = \emptyset$.
By Lemma \ref{exp Tn}, we know that $\mathbb{E}(|\mathcal{T}_n(\omega)| ) \ll q^{(-1/11)  \max \{ \deg n, M \} }$.
It then follows from Proposition \ref{K-dsv 1} that
$$
\mathbb{P}( \{ \mathcal{T}_n(\omega) \text{ contains a $12$-d.s.v. for some } n \in \mathbb{F}_q[t] \} ) = o_M(1).
$$
Therefore, our claim holds for vectorial sunflowers of this type.

Case 2. $|I \cap \{ 1, 2, 3 \}| = 1$. Without loss of generality, assume
$I \cap \{ 1, 2, 3 \} = \{1\}$. Let $l_1$ denote the common first coordinate.
If $\mathcal{T}_n(\omega)$ contains a vectorial sunflower of $12$ petals of type $I$ for some $n$, then there is a $12$-d.s.v.
$\{ \overline{x_j} \}_{1 \leq j \leq 12}$, where for each $j$ we have
$\overline{x_j} = (x_{2j},x_{3j})$, $Set(\overline{x_j}) \subseteq \omega$,  and $x_{2j} + x_{3j}= n- l_1$. Let $r = n - l_1$. Then
$\mathcal{U}_{r}(\omega)$ contains a $12$-d.s.v. and we obtain via Lemma \ref{family} our claim for
vectorial sunflowers of this type.

Case 3. $|I \cap \{ 1, 4, 5, 6 \}| = 2$ or  $| I \cap \{ 5, 6, 7, 8 \} | = 2$ or  $ | I \cap \{ 1, 4, 7, 8 \} | = 2$.
Suppose $|I \cap \{ 1, 4, 5, 6 \} | = 2$ as the other two cases are similar.
We consider the following two essentially distinct subcases separately.

$i$) Suppose $I \cap \{ 1, 4, 5, 6 \} = \{ 1, 4\}$.
Let $l_1$ and $l_4$ denote the common first and forth coordinates, respectively.
If $\mathcal{T}_n(\omega)$ contains a vectorial sunflower of $12$ petals of type $I$ for some $n$,
then there
is a $12$-d.s.v.
$\{ \overline{x}_j \}_{1 \leq j \leq 12}$, where for each $j$ we have
$\overline{x}_j = (x_{5j},x_{6j})$, $Set(\overline{x}_j) \subseteq \omega$,  and
$l_1 + l_4 = x_{5j} + x_{6j}$. Thus, for $r = l_1 + l_4$,
$\mathcal{U}_{r}(\omega)$ contains a $12$-d.s.v. and
we obtain via Lemma \ref{family}
our claim for vectorial sunflowers of this type. We can argue in a similar manner
if $I \cap \{ 1, 4, 5, 6 \} = \{ 5, 6\}$.

$ii$) Suppose $I \cap \{ 1, 4, 5, 6 \} = \{ 1, 5\}$.
Let $l_1$ and $l_5$ denote the common first and fifth coordinates, respectively.
If $\mathcal{T}_n(\omega)$ contains a vectorial sunflower of $12$ petals of type $I$ for some $n$, then
there is a $12$-d.s.v.
$\{ \overline{x}_j \}_{1 \leq j \leq 12}$, where for each $j$ we have
$\overline{x}_j = (x_{4j},x_{6j})$, $Set(\overline{x}_j) \subseteq \omega$,  and
$l_1 + x_{4j} = l_5 + x_{6j}$. Let $r = l_5 - l_1 = x_{4j} - x_{6j}$.
Note we have $r \not =
0$, because if $l_1 = l_5$, then the equation
$x_1 + x_4 = x_5 + x_6$ forces $\{x_1, x_4 \} = \{x_5, x_6 \}$, which is a contradiction. Thus,
$\mathcal{V}_{r}(\omega)$ contains a $12$-d.s.v. and we obtain via Lemma \ref{family}
our claim for vectorial sunflowers of this type. The remaining cases of $|I \cap \{ 1, 4, 5, 6 \}| = 2$
can be treated in a similar manner.

Case 4. $I = \{ 1, 2, 3, 4, 5, 6, 7, 8 \}$.
This is the trivial case. If two vectors have the same
$i$-th coordinate for all $i \in I$, then they are
the same vector. Thus, in particular $\mathcal{T}_n(\omega)$ does not
contain vectorial sunflowers of $12$ petals of this type.

Case 5. $| I \cap \{ 1,4,5,6,7,8 \} | = 1$. Without loss of generality suppose that
$ I \cap \{ 1,4,5,6,7,8 \}  = 1$.
Let $l_1$ denote the common first coordinate.
If $\mathcal{T}_n(\omega)$ contains a vectorial sunflower of $12$ petals of type $I$ for some $n$,
then there is a $12$-d.s.v.
$\{ \overline{x}_j \}_{1 \leq j \leq 12}$, where for each $j$ we have
$\overline{x}_j = (x_{4j},x_{5j}, x_{6j}, x_{7j}, x_{8j})$ , $Set(\overline{x}_j) \subseteq \omega$, and
$x_{5j} + x_{6j} = x_{7j} + x_{8j} = l_1 + x_{4j} $. Let $r = l_1$. Then
$\mathcal{W}_{r}(A)$ contains a $12$-d.s.v. and we obtain via Lemma \ref{family} our claim for
vectorial sunflowers of this type.
\end{proof}

We remark that in order for our argument to prove Propositions \ref{lower bound Q} and \ref{upper bound T_n} to work,
we needed the expectation of $\mathcal{Q}_n$ to go to infinity as $\deg n \rightarrow \infty$, while $\Delta(\mathcal{Q}_n)$
and the expectation of $\mathcal{T}_n$ to tend to $0$, and also the expectations of $\mathcal{U}_r$,
$\mathcal{V}_r$, and $\mathcal{W}_r$ to tend to $0$ as $\deg r \rightarrow \infty$. Our value of $\gamma = \frac{7}{11}$, similarly as in
\cite{C}, was chosen because it satisfies all of these conditions, and it also simplifies certain calculations. We note that it is
certainly not the only possible value for the method to prove Theorem \ref{main 1} to work.

\begin{proof}[Proof of Theorem \ref{main 1}]
By Proposition \ref{lower bound Q}, we have for $\omega \in \Omega$  with probability $1$  that
$$
|\mathcal{Q}_n(\omega)| \gg q^{\frac{1}{11} \deg n}
$$
for all $n \in \mathbb{F}_q[t] \backslash \{ 0 \}$ with $\deg n$ sufficiently large.
By Proposition \ref{upper bound T_n}, we know there exists $M$ sufficiently large
such that
$$
\mathbb{P}(\{ |\mathcal{T}_n(\omega)| \leq 10^{28}
\text{ for all } n \in \mathbb{F}_q[t] \}) \geq 1/2.
$$
Therefore, we deduce that there exists some $\omega_0 \in \Omega$ such that
$$
|\mathcal{Q}_n(\omega_0)| \gg q^{\frac{1}{11} \deg n }
$$
and
$$
|\mathcal{T}_n(\omega_0)| \leq 10^{28}
$$
for all $n \in \mathbb{F}_q[t] \backslash \{ 0 \}$ with $\deg n$ sufficiently large.
Thus we obtain our result by the argument given in the paragraph after ~(\ref{ineq 1}).
\end{proof}

\section{Proof of Theorem \ref{main 2}}
\label{sec proof of main 2}
Let $\varepsilon > 0$ be sufficiently small.
In this section, we consider the probability space \linebreak
$\mathcal{S}_M( \gamma ; S \ \text{mod } n_0 )$, where we let
$$
\gamma = \frac{2}{3} + \frac{\varepsilon}{9 + 9 \varepsilon},
$$
and let $S$ to be a non-empty subset of $\mathbb{G}_N$
satisfying the conditions of Corollary \ref{cor to thm}.
The basic strategy is as follows.
We use the Borell-Cantelli Lemma (Theorem \ref{Borel-C}) to show that in the probability space
$\mathcal{S}_M( \gamma ; S \ \text{mod } n_0 )$,
``most'' of the sequences, in other words with probability $1$, has ``many'' representations of $n$ as a sum of four of its elements,
where one of the four elements has degree less than or equal to $(\varepsilon \deg n)$,
for all $n \in \mathbb{F}_q[t]$ with $\deg n $ sufficiently large. We then show that out of these sequences, there exists
a sequence such that even after removing some of its elements to make it a Sidon sequence, it still has at least one
of the representations of $n$ left for each $n \in \mathbb{F}_q[t]$ with $\deg n $ sufficiently large.

For each $n \in \mathbb{F}_q[t]$, we consider the following collection of sets
$$
\mathcal{R}_n = \{ \theta = \{ x_1, x_2, x_3, x_4 \} : \theta \mbox{ satisfies Cond}(\mathcal{R}_n)  \},
$$
where
\begin{eqnarray}
\text{Cond} (\mathcal{R}_n)
=
\begin{cases}
x_1 + x_2 +  x_3 + x_4 = n, \\
\min \{ \deg x_1, \deg x_2, \deg x_3, \deg x_4 \} \leq \varepsilon \deg n,\\
x_i \not \equiv x_j \ (\text{mod }n_0) \mbox{ for } 1 \leq i < j \leq 4.
\end{cases}
\end{eqnarray}
Given a sequence of polynomials $\omega$, we let
$$
\mathcal{R}_n(\omega) = \{ \theta \in \mathcal{R}_n : \theta \subseteq \omega \}.
$$

We define
$$
\mathcal{B}_n = \{ \overline{x} = (x_1, x_2, x_3, x_4, x_5, x_6, x_7) : \overline{x} \text{ satisfies Cond} (\mathcal{B}_n)  \},
$$
where
\begin{eqnarray}
\text{Cond} (\mathcal{B}_n)
=
\begin{cases}
\{ x_1, x_2, x_3, x_4 \} \in \mathcal{R}_n, \\
x_1 + x_5 = x_6 + x_7, \ \ \  \{x_1, x_5\} \not =  \{x_6, x_7 \}, \\
x_1 \equiv x_6 \ (\text{mod }n_0), \ x_5 \equiv x_7 \ (\text{mod }n_0).
\end{cases}
\end{eqnarray}
We also let
$$
\mathcal{B}_n(\omega) = \{ \overline{x} \in B_n : \text{Set}(\overline{x}) \subseteq \omega \}.
$$

The \textit{Sidon lifting process} of a sequence $\omega$ consists of removing from
$\omega$, those elements $a \in \omega$ such that there exist $b,c,d \in \omega$
with $a + b = c + d$ and $\{a, b\} \not =  \{c , d \}$.
We denote by $\omega_{Sidon}$ the resulting Sidon sequence obtained by applying this process to $\omega$.

By a similar argument as in the paragraph before ~(\ref{ineq 1}), we see that
$|\mathcal{B}_n(\omega)|$ is an upper bound for the number of representations counted in $\mathcal{R}_n(\omega)$
that are destroyed in the Sidon lifting process of $\omega$. Thus, we obtain
\begin{equation}
\label{ineq 2}
|\mathcal{R}_n(\omega_{Sidon})| \geq |\mathcal{R}_n(\omega)| - |\mathcal{B}_n(\omega)|.
\end{equation}
Therefore, Theorem \ref{main 2} is established if we can prove that there
exists a sequence $\omega_0$ such that for any $n \in \mathbb{F}_q[t] \backslash \{ 0 \}$ with
$\deg n$ sufficiently large, we have $|\mathcal{R}_n(\omega_0)| \gg q^{\delta \deg n}$
for some $\delta > 0$, and $|\mathcal{B}_n(\omega_0)| \ll 1$.
We show that in some sense there are many sequences satisfying the former
condition, and then we show it is also the case for the latter condition. These tasks are accomplished in Propositions
\ref{lower bound R} and \ref{upper bound B_n}. We then prove that there exist sequences
satisfying both conditions.
Before we get into the proofs of these propositions,
we list three useful estimates. However, we postpone their proofs to Section \ref{calculations 2}.

\begin{lem}
\label{E R_n}
We have that
$$
\mathbb{E}(| \mathcal{R}_n(\omega)|) \gg q^{ \frac{2 \varepsilon^2}{9 + 9 \varepsilon} \deg n},
$$
for $n \in \mathbb{F}_q[t] \backslash \{ 0 \}$ with $\deg n$ sufficiently large.
\end{lem}

Recall from ~(\ref{def Delta}), the definition of $\Delta (\cdot)$.
\begin{prop}
\label{Delta R_n}
We have that
$$
\Delta( \mathcal{R}_n) \ll q^{\frac{- 3 \varepsilon + 2 \varepsilon^2}{9 + 9 \varepsilon} \deg n},
$$
for $n \in \mathbb{F}_q[t] \backslash \{ 0 \}$ with $\deg n$ sufficiently large.
\end{prop}

\begin{lem}
\label{exp Bn}
We have that
$$
\mathbb{E}(| \mathcal{B}_n(\omega)|) \ll q^{ - \frac{ \varepsilon^2}{18} \max \{ \deg n, M \} }.
$$
\end{lem}

We now prove the following proposition.
\begin{prop}
\label{lower bound R}
We have that
$$
\mathbb{P} ( \{ |\mathcal{R}_n(\omega)| \gg q^{\frac{2 \varepsilon^2}{9 + 9 \varepsilon} \deg n} \} ) = 1
$$
for $n \in \mathbb{F}_q[t] \backslash \{ 0 \}$ with $\deg n$ sufficiently large.
\end{prop}
\begin{proof}
We apply Theorem \ref{Janson} to $\mathcal{F} = \mathcal{R}_n$
and $Y = |\mathcal{R}_n(\omega)| = | \{ \theta \in \mathcal{R}_n : \theta \subseteq  \omega \}|$,
where $\omega$ is a random sequence in $\mathcal{S}_M( \gamma ; S \ \text{mod } n_0 )$.
We proved that $\mu = \mu_n = \mathbb{E}(\mathcal{R}_n(\omega)) \gg q^{\frac{ 2\varepsilon^2}{9 + 9 \varepsilon}}$ in Lemma \ref{E R_n},
and
$\Delta(\mathcal{R}_n) \ll q^{\frac{- 3 \varepsilon + 2\varepsilon^2}{9 + 9 \varepsilon}}$ in Proposition \ref{Delta R_n}.
Hence for $\deg n $ sufficiently large, we have $\Delta (R_n) < \mu_n$. Then, Theorem \ref{Janson} implies that
$$
\mathbb{P}( |\mathcal{R}_n(\omega)| \leq \mu_n/2) \leq \exp (- \mu_n/16).
$$
Therefore, we obtain that for some $C, C'>0$ and $T\in \mathbb{N}$, we can write
\begin{eqnarray}
\sum_{n \in \mathbb{F}_q[t]} \mathbb{P}(|R_n(\omega)| \leq \mu_n/2)
&<& C' + \sum_{j = T}^{\infty} (q^{j+1}-q^j) \exp(- C q^{\frac{ 2\varepsilon^2}{9 + 9 \varepsilon} j})
\notag
\\
&<& C' + (q - 1) \sum_{j = T}^{\infty} q^j \exp(- C q^{\frac{ 2\varepsilon^2}{9 + 9 \varepsilon} j})
\notag
\\
&<& \infty.
\notag
\end{eqnarray}
Thus, Theorem \ref{Borel-C} implies that with probability $1$, we have
$|\mathcal{R}_n(\omega)| > \mu_n/2 \gg q^{\frac{ 2\varepsilon^2}{9 + 9 \varepsilon} \deg n}$ for all $n \in \mathbb{F}_q[t] \backslash \{ 0 \}$ with $\deg n$
sufficiently large.
\end{proof}

For each $r \in \mathbb{F}_q[t]$, we define the following families of vectors.
\begin{eqnarray}
\label{family 2}
\mathcal{U}_{r} &=& \{ \overline{x} = (x_1, x_2) : x_1 + x_2 = r, x_1 \not = x_2 \},
\\
\mathcal{U}'_{r} &=& \{ \overline{x} = (x_1, x_2, x_3) : x_1 + x_2 + x_3 = r, x_i \not = x_j \ (i \not = j) \},
\notag
\\
\mathcal{V}_{r} &=& \{ \overline{x} = (x_1, x_2) : x_1 - x_2 = r, x_1 \not = x_2 \},
\notag
\\
\mathcal{V}'_{r} &=& \{ \overline{x} = (x_1, x_2, x_3) : x_1 + x_2 - x_3 = r, x_i \not = x_j \ (i \not = j) \}.
\notag
\end{eqnarray}

\begin{lem}
\label{family bound 2}
We have the following bounds on the expectations.
\newline
\newline
i) $\mathbb{E}(|\mathcal{U}_{r}(\omega)|) \ll q^{-\frac{1}{3} \max \{ \deg r, M \} }$.
\newline
ii) $\mathbb{E}(|\mathcal{U}'_{r}(\omega)|) \ll q^{-\frac{\varepsilon}{6} \max \{ \deg r, M \} }$.
\newline
iii) $\mathbb{E}(|\mathcal{V}_{r}(\omega)|) \ll q^{ - \frac{1}{3} \max \{ \deg r, M \} }$.
\newline
iv) $\mathbb{E}(|\mathcal{V}'_{r}(\omega)|) \ll q^{ - \frac{\varepsilon}{6} \max \{ \deg r, M \} }$.
\end{lem}
\begin{proof}
Since the details of the proof is similar to that of Lemma \ref{family bound} and \cite[Lemma 6.6]{C}, we omit the proof here.
\end{proof}

\begin{lem}
\label{family bound 2'}
Let $K \in \mathbb{N}$ and $K > 18 / \varepsilon^2$. Then for any of the four families
$\mathcal{F}_r$ in ~(\ref{family 2}), we have
$$
\mathbb{P}(\{ \mathcal{F}_r(\omega) \mbox{ contains a }K\mbox{-d.s.v. for some } r \in \mathbb{F}_q[t] \} ) = o_M(1).
$$
\end{lem}
\begin{proof}
For any of the fours choices of $\mathcal{F}_r$,
Lemma \ref{family bound 2} shows that $\mathbb{E}(|\mathcal{F}_r(\omega)|) \ll q^{- \frac{\varepsilon}{6} \max \{ \deg r, M \} }$.
Thus, the result follows by Proposition \ref{K-dsv 1}.
\end{proof}

We prove the following proposition, which is one of the main ingredients to prove Theorem \ref{main 2}.
\begin{prop}
\label{upper bound B_n}
Let $K \in \mathbb{N}$ and $K > 18 / \varepsilon^2$. We have that
$$
\mathbb{P}( \{ |\mathcal{B}_n(\omega)| \leq 7!((7^2 - 7+1) K )^7 \text{ for all } n \in \mathbb{F}_q[t] \} ) \geq 1 - o_M(1).
$$
\end{prop}
\begin{proof}
We claim the following statement:

\textbf{Claim.} Let $K$ be a positive integer such that $K > 18/ \varepsilon^2$.
Then with probability $1 - o_M(1)$, $\mathcal{B}_n(\omega)$ does not contain vectorial sunflower of $K$ petals for any $n \in \mathbb{F}_q[t]$.

Assuming the claim holds, we can apply Corollary \ref{cor sun and bound} to obtain
$$
\mathbb{P}( \{ |\mathcal{B}_n(\omega)| \leq 7!((7^2 - 7+1) K )^7  \text{ for all } n \in \mathbb{F}_q[t] \} ) \geq 1 - o_M(1).
$$
Thus, we see that proving the above claim is sufficient to obtain our result.
We prove it for distinct possible types $I \subseteq \{ 1, 2, 3, 4, 5, 6, 7 \}$
of vectorial sunflowers in $\mathcal{B}_n(\omega)$. We consider various cases in a similar manner as in Proposition \ref{upper bound T_n}
or \cite[Proposition 5.2]{C}.

By the definition of $\mathcal{B}_n$, it is easy to see that there can not
be a vectorial sunflower of type $I$, where
$| I \cap \{ 1, 2, 3, 4 \} | = 3$ or $| I \cap \{ 1, 5, 6, 7 \} | = 3$.
Suppose $| I \cap \{ 1, 2, 3, 4 \} | = 4$, then we have $| I \cap \{ 1, 5, 6, 7 \} | \geq 1$.
Similarly, if $| I \cap \{ 1, 5, 6, 7 \} | = 4$, then we have $| I \cap \{ 1, 2, 3, 4 \} | \geq 1$.
Therefore, the following six cases cover every possibility:
Case 1. $I = \emptyset$, Case 2. $|I \cap \{ 1, 2, 3, 4 \}| = 1$,
Case 3. $|I \cap \{ 1, 2, 3, 4 \}| = 2$, Case 4. $|I \cap \{ 1, 5, 6, 7 \}| = 1$,
Case 5. $|I \cap \{ 1, 5, 6, 7 \}| = 2$, and Case 6. $I = \{ 1, 2, 3, 4, 5, 6, 7 \}$.

The argument to show that the above claim holds for each of the six cases is similar to the argument
employed in Proposition \ref{upper bound T_n} and \cite[Proposition 5.2]{C}. Thus, we omit verifying the remaining details.
\end{proof}

We remark that in order for our argument to prove Propositions \ref{lower bound R} and \ref{upper bound B_n} to work,
we needed the expectation of $\mathcal{R}_n$ to go to infinity as $\deg n \rightarrow \infty$, while $\Delta(\mathcal{R}_n)$
and the expectation of $\mathcal{B}_n$ to tend to $0$, and also the expectations of $\mathcal{U}_r$,
$\mathcal{U}'_r$, $\mathcal{V}_r$, and $\mathcal{V}'_r$ to tend to $0$ as $\deg r \rightarrow \infty$.
Our value of $\gamma = \frac{2}{3} + \frac{\varepsilon}{9 + 9 \varepsilon}$, similarly as in \cite{C}, was chosen because it satisfies all of these conditions,
and it also simplifies certain calculations.

\begin{proof}[Proof of Theorem \ref{main 2}]
By Proposition \ref{lower bound R}, we have for $\omega \in \Omega$ with probability $1$ that
$$
|\mathcal{R}_n(\omega)| \gg q^{\frac{2 \varepsilon^2}{9 + 9 \varepsilon} \deg n}
$$
for all $n \in \mathbb{F}_q[t] \backslash \{ 0 \}$ with $\deg n$ sufficiently large.
By Proposition \ref{upper bound B_n}, we know there exists $M$ sufficiently large
such that
$$
\mathbb{P}(\{ |\mathcal{B}_n(\omega)| \leq 7!((7^2 - 7+1) K )^7
\text{ for all } n \in \mathbb{F}_q[t] \}) \geq 1/2.
$$
Therefore, we deduce that there exists some $\omega_0 \in \Omega$ such that
$$
|\mathcal{R}_n(\omega_0)| \gg q^{\frac{2 \varepsilon^2}{9 + 9 \varepsilon} \deg n}
$$
and
$$
|\mathcal{B}_n(\omega_0)| \leq 7!((7^2 - 7+1) K )^7
$$
for all $n \in \mathbb{F}_q[t] \backslash \{ 0 \}$ with $\deg n$ sufficiently large.
Thus we obtain our result by the argument given in the paragraph after ~(\ref{ineq 2}).
\end{proof}

\section{Technical Lemmas}
\label{appendix}
In this section, we calculate bounds that were used to compute estimates essential in the proof of our main results.
We list the estimates used in Theorems \ref{main 1} and \ref{main 2} in Sections \ref{calculations 1}
and \ref{calculations 2}, respectively. We consider $q$ to be a fixed
number, and throughout these sections implicit constants in inequalities may depend on $q$ without any further
notice. Let $\mathcal{A}, \mathcal{B}$ be two disjoint subsets of $\mathbb{F}_q[t]$, and $\nu_x \in \mathbb{R} \ ( x \in \mathbb{F}_q[t])$.
In order to avoid clutter in the exposition, we use the following notation for summation
$$
\sum_{x \in \mathcal{A}} + \sum_{x \in \mathcal{B}} \nu_x := \sum_{x \in \mathcal{A}} \nu_x + \sum_{x \in \mathcal{B}} \nu_x.
$$
We also use the notation in a similar manner when we have more than two pairwise disjoint sets.
Recall from Section \ref{prelim} that we have set $\mathbb{P}(0 \in \omega) = 0$.
Thus, we use the convention throughout the remainder of the paper that for any $\mu \in \mathbb{R}$,
we let $q^{\mu \deg 0} = 0$. We also let $\deg 0 := - \infty$. Finally, we note that in this section we do not require any assumption on the characteristic of $\mathbb{F}_q$. Thus, in particular, the results of this section hold even when the characteristic of $\mathbb{F}_q$ is $2$ or $3$.

For $\alpha, \beta \in \mathbb{R}$ and $n \in \mathbb{F}_q[t]$, we define the following quantities,
$$
\sigma_{\alpha, \beta}(n) = \sum_{ \stackrel{x,y \in \mathbb{F}_q[t] \backslash \{ 0 \} }{ x + y = n } } q^{ - \alpha \deg x} q^{ - \beta \deg y}
= \sum_{x \in \mathbb{F}_q[t] \backslash \{ 0 \} } q^{ - \alpha \deg x} q^{ - \beta \deg (n-x)},
$$
and
$$
\sigma_{\alpha, \beta}(n; M) = \sum_{ \deg x> M } q^{ - \alpha \deg x} q^{ - \beta \deg (n-x)}.
$$
For each $r \in \mathbb{N}$, there are exactly $q^{r+1} - q^r$ polynomials in $\mathbb{F}_q[t]$ of degree $r$.
Thus given any $\gamma > 1$ and $R \in \mathbb{N}$, we have
\begin{equation}
\label{sum bound 1}
\sum_{\deg x > R} q^{- \gamma \deg x} = \sum_{r = R+1}^{\infty} \left( q^{r+1} - q^r \right) q^{- \gamma r}
= (q - 1) \sum_{r = R+1}^{\infty} q^{(1 - \gamma) r}
\ll_{\gamma} q^{- (\gamma - 1) R} .
\end{equation}
Similarly for any $\lambda \in \mathbb{R}$, $\lambda > -1$, and $R \in \mathbb{N}$, we have
\begin{equation}
\label{sum bound 2}
\sum_{\deg x < R} q^{\lambda \deg x} = \sum_{r = 0}^{ R - 1} \left( q^{r+1} - q^r \right) q^{\lambda r}
= (q - 1) \sum_{r = 0}^{R-1} q^{(\lambda + 1) r}
= (q-1) \frac{q^{(\lambda + 1)R} - 1}{q^{(\lambda + 1)} - 1} \ll_{\lambda} q^{(\lambda + 1) R}.
\end{equation}

We have the following useful lemmas.
\begin{lem}
\label{basic 1}
Suppose $\alpha, \beta \in \mathbb{R}$ satisfy
$0 < \alpha < 1, 0 < \beta < 1$, and $\alpha + \beta > 1$. Then, for any $n \in \mathbb{F}_q[t] \backslash \{ 0 \}$, we have the following estimates:

i) $\sigma_{\alpha, \beta}(n; M) \ll q^{- (\alpha + \beta -1) \max \{ \deg n, M \} }$.

ii) $\sigma_{\alpha, \beta}(n) \ll q^{- (\alpha + \beta -1) \deg n}$.

Here the implicit constants depend only on $\alpha$, $\beta$ and $q$.
If $n=0$, then we still have i), but we have
$\sigma_{\alpha, \beta}(n) \ll 1$ for ii).
\end{lem}
\begin{proof}
Since the estimate is trivial when $n=0$, we assume $n \not = 0$.
We only prove $i)$ as $ii)$ follows immediately from $i)$
by letting $M = -1$.
Suppose $\deg n \leq M$.
Since $\alpha + \beta > 1$, it
follows by ~(\ref{sum bound 1}) that
$$
\sigma_{\alpha, \beta}(n; M) = \sum_{\deg x > M} q^{ - \alpha \deg x } q^{- \beta \deg (n-x)}
= \sum_{\deg x > M} q^{- (\alpha + \beta) \deg x }
\ll q^{-(\alpha + \beta -1) M}. 
$$

Suppose $N_0 = \deg n > M$.
Since $ - \alpha > -1$ and $\alpha + \beta > 1$, it follows by ~(\ref{sum bound 1})
and ~(\ref{sum bound 2}) that
\begin{eqnarray}
\label{sig ineq 1}
&&
\\
\sigma_{\alpha, \beta}(n; M) &=&
 \sum_{M < \deg x  < N_0} 
+
\sum_{\deg x = N_0} 
+
\sum_{\deg x > N_0}   q^{ - \alpha \deg x} q^{- \beta \deg (n-x)}
\notag
\\
\notag
&\leq&
q^{- \beta N_0} \sum_{M < \deg x  < N_0} q^{ - \alpha \deg x}
+
\sum_{\deg x = N_0} q^{- \alpha \deg x} q^{- \beta \deg (n-x)}
+
\sum_{\deg x > N_0} q^{ - (\alpha + \beta ) \deg x}
\\
&\ll&
q^{- \beta N_0} q^{(1 - \alpha) N_0}
+
q^{- \alpha N_0}  \sum_{\deg x = N_0} q^{- \beta \deg (n+x)}
+
q^{-(\alpha + \beta -1) N_0}.
\notag
\end{eqnarray}
We now deal with the remaining sum on the right hand side of the inequality displayed above.
Given a degree $N_0$ polynomial $x$, we let
$x = c t^{N_0} + y$, where $c \in \mathbb{F}_q \backslash \{ 0\} $ and $y \in \mathbb{G}_{N_0}$.
If the leading coefficient of $n$ is $c' \in \mathbb{F}_q \backslash \{ 0\}$, then
we see that $\{ n + (- c') t^{N_0} +  y : y \in \mathbb{G}_{N_0} \} = \mathbb{G}_{N_0}$.
Thus we obtain the following bound by ~(\ref{sum bound 2}),
\begin{eqnarray}
\sum_{\deg x = N_0} q^{- \beta \deg (n+x)}
&=&
\sum_{ a \in \mathbb{F}_q \backslash \{ 0, -c'\} } \  \sum_{\deg y < N_0} q^{- \beta \deg (n + at^{N_0} + y )}
+
\sum_{\deg y < N_0} q^{- \beta \deg (n + (-c') t^{N_0} + y )}
\notag
\\
&=&
(q-2) \sum_{\deg y < N_0} q^{- \beta \deg n } + \sum_{\deg z < N_0} q^{- \beta \deg z}
\notag
\\
&\ll&
(q-2) q^{N_0 - \beta N_0 } + q^{(1 - \beta) N_0}.
\notag
\end{eqnarray}
Therefore, it follows from ~(\ref{sig ineq 1}) that when $\deg n > M$, we have
$$
\sigma_{\alpha, \beta}(n; M) \ll q^{- (\alpha + \beta - 1) \deg n}.
$$
\end{proof}

\begin{lem}
\label{basic 3}
Suppose $\phi,\kappa \in \mathbb{R}$ satisfy
$0 < \phi < 1, 0< \kappa < 1$, and $\phi + \kappa > 1$. Let $r \in \mathbb{F}_q[t]$. Then we have that
$$
\sum_{ \deg x > M } q^{- \phi \deg x}
q^{- \kappa \max \{ \deg (r + x), M \} }
\ll
q^{(1  - \phi - \kappa ) \max \{ \deg r, M \}  }.
$$
\end{lem}
\begin{proof}
We consider the two cases, $\deg r \leq M$ and $\deg r > M$, separately.
Suppose $\deg r \leq M$. Then we have
$ \deg (r + x) = \deg x > M$.  Since
$\phi + \kappa > 1$, we obtain the following bound by ~(\ref{sum bound 1}),
\begin{equation}
\label{first bound for Wr}
\sum_{ \deg x > M } q^{- \phi \deg x}
q^{- \kappa \max \{ \deg (r + x), M \} }  \ll \sum_{ \deg x > M } q^{ ( - \phi - \kappa ) \deg x} \ll q^{(1  - \phi - \kappa )M}.
\end{equation}

Next, suppose $\deg r > M$. We split and simplify the sum as follows,
\begin{eqnarray}
&& \sum_{ \deg x > M } q^{- \phi \deg x} q^{ - \kappa \max \{ \deg (r + x), M \} }
\notag
\\
&\leq& \sum_{\deg x < \deg r} + \sum_{ \deg x = \deg r} + \sum_{ \deg x > \deg r } q^{- \phi \deg x} q^{- \kappa \max \{ \deg (r + x), M \} }
\notag
\\
&=&
q^{- \kappa \deg r } \sum_{\deg x < \deg r} q^{- \phi \deg x}
+ \sum_{\deg x = \deg r}  q^{- \phi \deg x} q^{- \kappa \max \{ \deg (r + x), M \} }
\notag
\\
&& \phantom{000011111222333333444444444445555555555666666666} +\sum_{\deg x > \deg r}  q^{(- \kappa - \phi) \deg  x }.
\notag
\end{eqnarray}
Note $ - \phi > -1$ and $ \kappa + \phi > 1$. Thus by applying ~(\ref{sum bound 2}) and
~(\ref{sum bound 1}) to the first sum and the third sum, respectively, we see that these sums are
bounded by
\begin{eqnarray}
\label{W sum 1}
\ll q^{(1 - \kappa + \phi) \deg r}.
\end{eqnarray}

We now deal with the remaining second sum. Let $N_0 = \deg x = \deg r $.
Write a degree $N_0$ polynomial $x$ as
$x = c t^{N_0} + y$, where $c \in \mathbb{F}_q \backslash \{ 0\} $ and $y \in \mathbb{G}_{N_0}$.
Given $z \in \mathbb{F}_q[t]$ we define $lead[z] \in \mathbb{F}_q \backslash \{ 0\}$ to be
the coefficient of $t^{\deg z}$ in $z$. By separating the cases $c \not = - lead[r]$ and $c = - lead[r]$, we obtain the following bound for the sum in question,
\begin{eqnarray}
\label{W sum 2'}
&&
\sum_{ \deg x = \deg r = N_0 } q^{- \phi \deg x}
q^{- \kappa \max \{ \deg (r + x), M \} }
\\
&=&
q^{- \phi N_0} \sum_{ c \in \mathbb{F}_q[t] \backslash \{ 0\} } \sum_{\deg y < N_0} q^{ - \kappa \max \{  \deg (r + ct^{N_0} + y ) , M \}}
\notag
\\
&=&
(q-2) q^{- \phi N_0} \sum_{\deg y < N_0} q^{ - \kappa N_0 } + q^{- \phi N_0} \sum_{\deg z < N_0} q^{ - \kappa \max \{ \deg z , M \}}
\notag
\\
&=&
(q-2) q^{- \phi N_0} q^{(1  - \kappa) N_0 }
+ q^{- \phi N_0} \sum_{ \deg z \leq M } q^{- \kappa M}
+ q^{- \phi N_0} \sum_{ M < \deg z < N_0}  q^{- \kappa \deg z}.
\notag
\end{eqnarray}
Since $- \kappa > -1$, we can apply ~(\ref{sum bound 2}) to the third sum,
and obtain the following bound for ~(\ref{W sum 2'}),
\begin{eqnarray}
\label{W sum 2}
\ll q^{(1 - \phi - \kappa ) N_0} + q^{- \phi N_0} q^{(1 - \kappa) M} \ll q^{(1 - \phi - \kappa ) N_0}.
\end{eqnarray}
The last inequality follows, because $1 - \kappa > 0$ and $N_0 > M.$
Therefore, 
we obtain that
$$
\sum_{ \deg x > M } q^{- \phi \deg x} q^{- \kappa \max \{ \deg (r + x), M \} }  \ll q^{(1 - \phi - \kappa) \deg r}
$$
when $\deg r > M$.
\end{proof}

We also have the following lemma, which we make use of only in Section \ref{calculations 1}.
\begin{lem}
\label{basic 2}
Given any polynomials $a, b \in \mathbb{F}_q[t] \backslash \{ 0 \}$ and $1/2 < \gamma < 2/3$, we have
$$
\sum_{x \in \mathbb{F}_q[t] \backslash \{ 0 \} } q^{- \gamma \deg x} q^{- \gamma \deg (x+a) } q^{ (1 - 2 \gamma) \deg (x+b) }
\ll
q^{( 1 - 2 \gamma) (\deg a + \deg b) }.
$$
If $a=0$ and $b \in \mathbb{F}_q[t] \backslash \{ 0 \}$, then we have
$$
\sum_{x \in \mathbb{F}_q[t]  \backslash \{ 0 \}  } q^{- \gamma \deg x} q^{- \gamma \deg (x+a) } q^{ (1 - 2 \gamma) \deg (x+b) }
\ll
q^{( 1 - 2 \gamma) \deg b }.
$$
\end{lem}
\begin{proof}
We consider the following three cases separately: $\deg a < \deg b, \deg a > \deg b$ and $\deg a = \deg b$.
We only present the details of the computation for the cases $\deg a < \deg b$ and $\deg a = \deg b$. The result
for the case $\deg a > \deg b$ can be verified in a similar manner as the case $\deg a < \deg b$.
First, suppose $0 \leq \deg a < \deg b$. We split the sum and simplify it in the following manner,
\begin{eqnarray}
\label{first sum in lem 5}
&&\sum_{x \in \mathbb{F}_q[t]  \backslash \{ 0 \} } q^{- \gamma \deg x} q^{- \gamma \deg (x+a) } q^{ (1 - 2 \gamma) \deg (x+b) }
\\
&=&
\sum_{ \deg x < \deg b}
+ \sum_{ \deg x = \deg b}
+ \sum_{ \deg x > \deg b}
q^{- \gamma \deg x} q^{- \gamma \deg (x+a) } q^{ (1 - 2 \gamma) \deg (x+b) }
\notag
\\
&=&
q^{ (1 - 2 \gamma) \deg b} \sum_{ \deg x < \deg b} q^{- \gamma \deg x - \gamma \deg (x+a) }
+
q^{- 2 \gamma \deg b} \sum_{ \deg x = \deg b}  q^{ (1 - 2 \gamma) \deg (x+b) }
\notag
\\
&&\phantom{12345678901234567890111111111112222222222223333333333}+
 \sum_{ \deg x > \deg b} q^{( 1 - 4 \gamma ) \deg x} 
\notag
\\
\notag
\\
&\leq&
q^{ (1 - 2 \gamma) \deg b} \sum_{x \in \mathbb{F}_q[t]  \backslash \{ 0 \} } q^{- \gamma \deg x - \gamma \deg (x+a) }
+
q^{- 2 \gamma \deg b} \sum_{ \deg z < 1 + \deg b }  q^{ (1 - 2 \gamma) \deg z }
\notag
\\
&&\phantom{12345678901234567890111111111112222222222223333333333}+
 \sum_{ \deg x > \deg b} q^{( 1 - 4 \gamma ) \deg x}. 
\notag
\end{eqnarray}
Note we have $2 \gamma > 1$, $1 - 2 \gamma > -1$ and $ 4 \gamma - 1 > 1$.
Thus, we apply Lemma \ref{basic 1} to the first sum, ~(\ref{sum bound 2}) to the second sum, and
~(\ref{sum bound 1}) to the third sum, in order to estimate the final expression above.
Consequently, we obtain that the above expression is bounded by
\begin{eqnarray}
&\ll&
q^{ (1 - 2 \gamma) \deg b} q^{ (1 - 2 \gamma) \deg a }
+
q^{- 2 \gamma \deg b} q^{ (2 - 2 \gamma) \deg b }
+
q^{ (2 - 4 \gamma) \deg b }
\notag
\\
&\ll&
q^{(1 - 2 \gamma) \deg b} q^{ (1 - 2 \gamma) \deg a }.
\notag
\end{eqnarray}
Our conclusion when $a=0$ follows from a similar analysis as above, except that we have to use the following bound
obtained via ~(\ref{sum bound 1}) instead of Lemma \ref{basic 1} to bound ~(\ref{first sum in lem 5}),
$$
\sum_{x \in \mathbb{F}_q[t]  \backslash \{ 0 \} } q^{- \gamma \deg x - \gamma \deg (x+a) }
\ll
\sum_{x \in \mathbb{F}_q[t]  \backslash \{ 0 \} } q^{- 2 \gamma \deg x} \ll 1.
$$

Next, suppose $\deg a = \deg b$.
We split the sum and simplify it in the following manner,
\begin{eqnarray}
\label{eqn deg a = deg b}
&&\sum_{x \in \mathbb{F}_q[t] } q^{- \gamma \deg x} q^{- \gamma \deg (x+a) } q^{ (1 - 2 \gamma) \deg (x+b) }
\\
&=&
 \sum_{ \deg x < \deg b}
+
\sum_{ \deg x = \deg b}
+
\sum_{ \deg x > \deg b}  q^{- \gamma \deg x} q^{- \gamma \deg (x+a) } q^{ (1 - 2 \gamma) \deg (x+b) }
\notag
\\
&=&
q^{ (1 - 3 \gamma) \deg b} \sum_{ \deg x < \deg b} q^{- \gamma \deg x }
+
q^{- \gamma \deg b} \sum_{ \deg x = \deg b} q^{- \gamma \deg (x+a) } q^{ (1 - 2 \gamma) \deg (x+b) }
\notag
\\
&&\phantom{12345678901234567890111111111112222222222223333333333}+
\sum_{ \deg x > \deg b} q^{( 1 - 4 \gamma ) \deg x}.
\notag
\end{eqnarray}
Note we have $- \gamma > -1$ and $4 \gamma -1 > 1$. Thus, we apply
~(\ref{sum bound 2}) to the first sum, and
~(\ref{sum bound 1}) to the third sum, and obtain that
the first and the third terms of the final expression above are bounded by
$$
\ll q^{(2 - 4 \gamma ) \deg b} = q^{(1 - 2 \gamma ) (\deg a + \deg b) }.
$$
The remaining second sum requires more work to estimate.
It is clear that the set of polynomials of degree $(\deg b)$ can be expressed as
$$
\{ c t^{\deg b} + y : c \in \mathbb{F}_q \backslash \{ 0\}, y \in \mathbb{G}_{\deg b}\}.
$$
Let $c \in \mathbb{F}_q \backslash \{ 0 \}$, and let
$lead[a]$ be the leading coefficient of $a$. Then, it follows that
\begin{eqnarray}
\label{a + c t}
(a + c t^{\deg b}) + \mathbb{G}_{\deg b} =
\begin{cases}
\mathbb{G}_{\deg b}, &\mbox{if } c = - lead[a], \\
(lead[a] + c) t^{\deg b} + \mathbb{G}_{\deg b}, & \mbox{otherwise. }
\end{cases}
\end{eqnarray}
We also have a similar statement if we replace $a$ by $b$.
We consider the second sum in the final expression of ~(\ref{eqn deg a = deg b}) in two separate cases,
$lead[a] \not = lead[b]$ and $lead[a] = lead[b]$.

Suppose $lead[a] \not = lead[b]$.
We utilize ~(\ref{a + c t}) to simplify the sum in question in the following manner,
\begin{eqnarray}
&& \sum_{ \deg x = \deg b} q^{- \gamma \deg (x + a) } q^{ (1 - 2 \gamma) \deg (x+b) }
\notag
\\
&=&
\sum_{c \in \mathbb{F}_q \backslash \{ -lead[a], -lead[b], 0 \}} \
\sum_{ \deg y < \deg b } q^{- \gamma \deg (a + c t^{\deg b} + y)}
q^{ (1 - 2 \gamma) \deg (b + c t^{\deg b} + y) }
\notag
\\
&& \phantom{1111111111222222}
+
\sum_{ \deg y < \deg b } q^{- \gamma \deg (a -lead[a] t^{\deg b} + y)}
q^{ (1 - 2 \gamma) \deg (b -lead[a] t^{\deg b} + y) }
\notag
\\
&& \phantom{1111111111222222222333333}+
\sum_{ \deg y < \deg b } q^{- \gamma \deg (a -lead[b] t^{\deg b} + y)}
q^{ (1 - 2 \gamma) \deg (b -lead[b] t^{\deg b} + y) }
\notag
\\
&=&
(q-3) q^{\deg b} q^{(1 - 3 \gamma) \deg b}
+
q^{ (1 - 2 \gamma) \deg b } \sum_{ \deg y < \deg b } q^{- \gamma \deg (a -lead[a] t^{\deg b} + y)}
\notag
\\
&& \phantom{111222233334445555555566666777778888}+
q^{ -  \gamma \deg b } \sum_{ \deg y < \deg b } q^{(1 -2 \gamma) \deg (b -lead[b] t^{\deg b} + y)}
\notag
\\
&\ll&
q^{(2 - 3 \gamma) \deg b}
+
q^{ (1 - 2 \gamma) \deg b } \sum_{ \deg z < \deg b } q^{- \gamma \deg z}
+ q^{ -  \gamma \deg b } \sum_{ \deg z < \deg b } q^{( 1- 2\gamma) \deg z}.
\notag
\end{eqnarray}
Since $- \gamma > -1$ and $1 - 2 \gamma > -1$, we obtain by ~(\ref{sum bound 2})
that the final expression above is bounded by
$$
\ll q^{(2 - 3 \gamma) \deg b}.
$$

On the other hand, suppose $lead[a] = lead[b]$. We simplify the sum in a similar manner as in the previous case,
\begin{eqnarray}
\label{la=lb}
&& \sum_{ \deg x = \deg b} q^{- \gamma \deg (x + a) } q^{ (1 - 2 \gamma) \deg (x+b) }
\\
&=&
\sum_{c \in \mathbb{F}_q \backslash \{ -lead[b], 0 \}} \
\sum_{ \deg y < \deg b } q^{- \gamma \deg (a + c t^{\deg b} + y)}
q^{ (1 - 2 \gamma) \deg (b + c t^{\deg b} + y) }
\notag
\\
&& \phantom{11111111112222222223333}
+
\sum_{ \deg y < \deg b } q^{- \gamma \deg (a -lead[b] t^{\deg b} + y)}
q^{ (1 - 2 \gamma) \deg (b -lead[b] t^{\deg b} + y) }
\notag
\\
&=&
(q-2) q^{\deg b} q^{(1 - 3 \gamma) \deg b}
+
\sum_{ \deg y < \deg b } q^{- \gamma \deg (a - lead[a] t^{\deg b} + y)}
q^{ (1 - 2 \gamma) \deg (b -lead[b] t^{\deg b} + y) }.
\notag
\end{eqnarray}
Let $g$ be the polynomial,
$$
g = b - lead[b] t^{\deg b} - ( a - lead[a] t^{\deg b}),
$$
where $\deg g < \deg b$.
By the change of variable $x = y + (a - lead[a] t^{\deg b})$,
we see that the final expression in ~(\ref{la=lb}) equals to
\begin{eqnarray}
\label{la = lb 2}
(q-2) q^{ (2 - 3 \gamma) \deg b}
+
\sum_{ \deg x < \deg b } q^{- \gamma \deg x }
q^{ (1 - 2 \gamma) \deg (x + g) }.
\end{eqnarray}

We have that if $\deg g \leq \deg x$, then $\deg (x+g) \leq \deg x$, and equivalently, $- \gamma \deg (x+g) \geq - \gamma \deg x$.
Since $1- 2 \gamma < 0$, we bound the sum in ~(\ref{la = lb 2}) as follows,
\begin{eqnarray}
&&\sum_{ \deg x < \deg b } q^{- \gamma \deg x } q^{ (1 - 2 \gamma) \deg (x + g) }
\notag
\\
&=&
\sum_{\deg x < \deg g} + \sum_{\deg g \leq \deg x < \deg b} q^{- \gamma \deg x } q^{ (1 - 2 \gamma) \deg (x + g) }
\notag
\\
&=&
\sum_{\deg x < \deg g} q^{- \gamma \deg x } q^{ (1 - 2 \gamma) \deg g }
+ \sum_{\deg g \leq \deg x < \deg b} q^{- \gamma \deg x } q^{ (1 - 2 \gamma) \deg (x + g) }
\notag
\\
&\ll&
\sum_{\deg x < \deg g} q^{- \gamma \deg x } q^{ (1 - 2 \gamma) \deg x }
+ \sum_{\deg g \leq \deg x < \deg b} q^{- \gamma \deg (x+g) } q^{ (1 - 2 \gamma) \deg (x + g) }
\notag
\\
&\ll&
\sum_{\deg x < \deg g} q^{ (1 - 3 \gamma) \deg x }
+ \sum_{\deg z < \deg b}  q^{ (1 - 3 \gamma) \deg z }.
\notag
\end{eqnarray}
Since $\gamma < 2/3$, we have $1 - 3 \gamma > -1$. Thus, by ~(\ref{sum bound 2}) we see that
the final expression above is bounded by
$$
\ll q^{(2 - 3 \gamma) \deg g} + q^{(2 - 3 \gamma) \deg b} \ll q^{(2 - 3 \gamma) \deg b}.
$$
Consequently, ~(\ref{la=lb}) is bounded by
$$
\ll q^{(2 - 3 \gamma) \deg b}.
$$
Therefore, in either case we obtain that ~(\ref{eqn deg a = deg b}) is bounded by
$$
\ll q^{(2 - 4 \gamma) \deg b} = q^{(1 - 2 \gamma) (\deg a + \deg b)}
$$
as desired.
\end{proof}

\section{Estimates in Theorem \ref{main 1} }
\label{calculations 1}

Recall in Section \ref{sec proof of main 1} we work in the probability space $\mathcal{S}_M( \gamma ; S \ \text{mod } n_0 )$,
where $\gamma = 7/ 11$ and $S$ is a non-empty subset of $\mathbb{G}_N$ satisfying the conditions of Theorem \ref{thm 2.1 in C}.

\begin{replemma}{E Q_n}
We have that
$$
\mathbb{E}(| \mathcal{Q}_n(\omega) |) \gg q^{\frac{1}{11} \deg n},
$$
for $n \in \mathbb{F}_q[t] \backslash \{ 0 \}$ with $\deg n$ sufficiently large.
\end{replemma}
\begin{proof}
By the definition of $\mathcal{Q}_n(\omega)$, we have
$$
\mathbb{E}(| \mathcal{Q}_n(\omega) |) = \sum_{ \{ x_1,  x_2,  x_3 \} \in Q_n } \mathbb{P}(x_1, x_2, x_3 \in \omega) \geq q^{ - 3 \gamma \deg n} |\mathcal{Q}'_n|,
$$
where
$$
\mathcal{Q}'_n = \{ \{ x_1, x_2, x_3 \} \in \mathcal{Q}_n : x_i \equiv S \ (\text{mod } n_0 ), \deg n = \deg x_i > M \}.
$$
By our choice of $S$, we know there exist distinct $s_1, s_2, s_3$ such that
$n \equiv s_1 + s_2 + s_3 \ (\text{mod } n_0)$. We fix such  $s_1, s_2, s_3$,
and write $x_i = s_i + n_0 y_i$. Let $l \in \mathbb{F}_q[t]$ be the polynomial such that $n - s_1 - s_2 - s_3 = l \ n_0$.
Then we have $|\mathcal{Q}'_n| \geq |\mathcal{Q}^*_n|$, where
$$
\mathcal{Q}^*_n = \{ \{ y_1, y_2, y_3 \} : y_1 + y_2 + y_3 = l, \deg y_i = \deg n - \deg n_0  \}.
$$
Given any element $a \in \mathbb{F}_q$, the number of solutions
$(a_1, a_2, a_3) \in \mathbb{F}_q \times \mathbb{F}_q \times \mathbb{F}_q$ such that $a_1 + a_2 + a_3 = a$
and $a_i \not = 0$ is greater than or equal to $(q-1)^2 - q$. We see this by allowing
$a_1$ and $a_2$ to be any elements in $\mathbb{F}_q\backslash \{ 0\}$,
which then uniquely determines $a_3$. This gives $(q-1)^2$ choices, but we
do not want to include cases when $a_3 = 0$, which only occurs
when $a_1 + a_2 = a$. There are only $q$ combinations of $a_1$ and $a_2$
such that $a_1 + a_2 = a$. Therefore, by considering the addition coordinatewise,
we obtain the following crude bound,
$$
|\mathcal{Q}^*_n| \geq \frac16 \left( (q-1)^2 - q \right)^{\deg y_i}  \geq \frac16 (q(q-3))^{\deg n - \deg n_0} \gg q^{2 \deg n}.
$$
Note the factor of $1/6$ is there to take care of possible over counting of the triplets.
Hence, we obtain our result,
$$
\mathbb{E}(|\mathcal{Q}_n(\omega)|) \geq q^{ - 3 \gamma \deg n} | \mathcal{Q}'_n| \gg q^{ (2 - 3 \gamma) \deg n} = q^{\frac{1}{11} \deg n}.
$$
\end{proof}

\begin{repproposition}{Delta Q_n}
We have that
$$
\Delta( \mathcal{Q}_n) \ll q^{-\frac{2}{11} \deg n},
$$
for $n \in \mathbb{F}_q[t] \backslash \{ 0 \}$ with $\deg n$ sufficiently large.
\end{repproposition}
\begin{proof}
Recall by $\theta \sim \theta'$, we mean that $\theta \cap \theta' \not = \emptyset$
and $\theta \not = \theta'$. Thus if $\theta \sim \theta'$ for $\theta, \theta' \in \mathcal{Q}_n$,
then it follows that $ | \theta \cap \theta' | = 1$, for otherwise $\theta = \theta'$.
Without loss of generality, let the common element be $x_1$.
We bound $\Delta( \mathcal{Q}_n)$ by applying Lemma \ref{basic 1} twice,
\begin{eqnarray}
\Delta(\mathcal{Q}_n)
&=&
\sum_{\stackrel{\theta, \theta' \in \mathcal{Q}_n}{\theta \sim \theta'}} \mathbb{P}(\theta, \theta' \subseteq \omega)
\notag
\\
&\ll&
\sum_{\substack{x_1, x_2,x_3,x'_2,x'_3 \in \mathbb{F}_q[t] \backslash \{ 0 \}\\ x_2+x_3=n-x_1 \\ x'_2+x'_3=n-x_1}}
q^{- \gamma (\deg x_1 + \deg x_2 + \deg x_3 + \deg x'_2 + \deg x'_3)}
\notag
\\
&\ll&
\sum_{ x_1 \in \mathbb{F}_q[t] \backslash \{ 0 \} } q^{- \gamma \deg x_1 }
\left( \sum_{\substack{x_2, x_3 \in \mathbb{F}_q[t] \backslash \{ 0 \}\\ x_2+x_3=n-x_1 }} q^{- \gamma (\deg x_2 + \deg x_3) } \right)^2
\notag
\\
&\ll&
\sum_{ x_1 \in \mathbb{F}_q[t] \backslash \{ 0 \} } q^{- \gamma \deg x_1 }
 q^{ ( 2 - 4 \gamma) \max \{ \deg (n - x_1), 0 \} }
\notag
\\
&\ll&
q^{- \gamma \deg n }
+
\sum_{ x_1 \in \mathbb{F}_q[t] \backslash \{ 0 \} } q^{- \gamma \deg x_1 } q^{(2 - 4 \gamma) \deg (n - x_1) }
\notag
\\
&\ll&
q^{- \gamma \deg n }
+
q^{(3 - 5 \gamma) \deg n}
\notag
\\
&\ll&
q^{(3 - 5 \gamma) \deg n}.
\notag
\end{eqnarray}
Note since $ 0 < 4 \gamma - 2 = 6/11 < 1$ and $5 \gamma - 2 = 13/11 > 1$, the sum after the third last inequality satisfies the
conditions required to apply Lemma \ref{basic 1}.
We also remark that the term $q^{- \gamma \deg n }$ in the third last inequality comes from the case $n - x_1 = 0$.
\end{proof}

\begin{replemma}{exp Tn}
We have that
$$
\mathbb{E}(| \mathcal{T}_n(\omega)|) \ll q^{-\frac{1}{11} \max \{ \deg n, M \} }.
$$
\end{replemma}

\begin{proof}
Recall from Section \ref{sec proof of main 1} that
$$
\mathcal{T}_n = \{ \overline{x} = (x_1, x_2, x_3, x_4, x_5, x_6, x_7,x_8) : \overline{x} \text{ satisfies Cond} (\mathcal{T}_n)  \},
$$
where
\begin{eqnarray}
\text{Cond} (\mathcal{T}_n)
=
\begin{cases}
\{ x_1, x_2, x_3 \} \in \mathcal{Q}_n, \\
x_1 + x_4 = x_5 + x_6 = x_7 + x_8, \ \ \  \{x_1, x_4\} \not =  \{x_5, x_6 \} \not =  \{x_7, x_8 \}, \\
x_1 \equiv x_5 \equiv x_7 \ (\text{mod }n_0), \ x_4 \equiv x_6\equiv x_8 \ (\text{mod }n_0).
\end{cases}
\notag
\end{eqnarray}
Suppose $\overline{x} = (x_1, x_2, x_3, x_4, x_5, x_6, x_7, x_8) \in \mathcal{T}_n$.
Then we know that each element of $\{x_1, x_2, x_3\}$ is
distinct by the definition of $ \mathcal{Q}_n$.
We have that the elements of $\{x_1, x_5, x_7 \}$ are distinct for the following reason.
Without loss of generality, suppose $x_1 = x_5$, then the equation $x_1 + x_4 = x_5 + x_6$ forces $x_4 = x_6$.
Hence, we have $\{x_1, x_4 \} = \{x_5, x_6 \}$, which is a contradiction.
By a similar argument, the elements of $\{x_4, x_6, x_8 \}$ are distinct.
We classify every situation where we have a repeated element amongst $\{ x_i\}_{1 \leq i \leq 8}$ by
considering the following two cases separately, $x_1 \equiv x_4 \ (\text{mod } n_0)$
and  $x_1 \not \equiv x_4 \ (\text{mod } n_0)$.

Case 1 : $x_1 \equiv x_4 \ (\text{mod } n_0)$. If $x_1 \in  \{ x_6, x_8 \}$, then we obtain contradiction
by the following argument.
Without loss of generality, suppose $x_1 = x_6$,
then the equation $x_1 + x_4 = x_5 + x_6$ forces $x_4 = x_5$.
Hence, we have $\{x_1, x_4 \} = \{x_5, x_6 \}$, which is a contradiction.
Thus, it follows that the only possible repetition of $x_1$ is $x_1 = x_4.$
By the definition of $ \mathcal{Q}_n$, we have $x_1 \not \equiv x_2, x_3 \ (\text{mod } n_0)$.
Since $x_1 \equiv x_i \ (\text{mod } n_0)$ for $4 \leq i \leq 8$, we have
$\{x_1, x_4, x_5, x_6, x_7, x_8 \} \cap \{ x_2, x_3 \} = \emptyset$.
We also know that $x_2 \not = x_3$. Therefore, $x_2$ and $x_3$ do not have a repetition.
If $x_4 = x_i$ for $5 \leq i \leq 8$, then the relations
$x_1 + x_4 = x_5 + x_6 = x_7 + x_8$ and $\{x_1, x_4\} \not =  \{x_5, x_6 \} \not =  \{x_7, x_8 \}$
yield contradiction. Thus, $x_4$ has no possible repetition other than $x_1 = x_4$.
We continue in a similar manner and verify that the only remaining possible repetitions are $x_5 = x_6$ and $x_7 = x_8$.
Note the entries of $\overline{x}$ can not have more than one of the three possible repetitions,
because otherwise it violates $\{x_1, x_4\} \not =  \{x_5, x_6 \} \not =  \{x_7, x_8 \}$.
Therefore, the possible subcases stemming from Case 1 are
$\{x_1, x_2, x_3, x_4, x_5, x_6, x_7, x_8 \}$ are distinct, and each of the following,
\begin{eqnarray}
\begin{cases}
\ x_1 = x_4 \text{ and } \{x_1, x_2, x_3, x_5, x_6, x_7, x_8 \} \text{ are distinct},
\notag
\\
\ x_5 = x_6 \text{ and } \{x_1, x_2, x_3, x_4, x_5, x_7, x_8 \} \text{ are distinct},
\notag
\\
\ x_7 = x_8 \text{ and } \{x_1, x_2, x_3, x_4, x_5, x_6, x_7\} \text{ are distinct}.
\notag
\end{cases}
\end{eqnarray}

Case 2 : $x_1 \not \equiv x_4 \ (\text{mod } n_0)$. In this case,
we have $\{x_1, x_5, x_7\} \cap \{x_4, x_6, x_8 \} = \emptyset$, and
consequently,  $\{x_1, x_4, x_5, x_6, x_7, x_8 \}$ are distinct.
We know that $x_1, x_5, x_7 \not \in \{ x_2, x_3 \}$, because they are in
different residue classes modulo $n_0$. Therefore, it follows
that $x_1, x_5, x_7$ do not have any repetitions.
Thus, we deduce that the only possible repetitions are of the form
$x_j \in \{ x_2, x_3 \}$, where $j \in \{ 4, 6, 8 \}$.
Without loss of generality, suppose $x_4 = x_2$. Then since $x_4 = x_2 \not \equiv x_3 \ (\text{mod }n_0)$,
we have that $x_3 \not = x_2, x_4, x_6, x_8$. It follows that
$\{x_1, x_2, x_3, x_5, x_6, x_7, x_8 \}$ are distinct.
We obtain similar conclusions for other cases.
Therefore, the possible subcases stemming from Case 2 are
$\{x_1, x_2, x_3, x_4, x_5, x_6, x_7, x_8 \}$ are distinct, and
each of the following situations, where there exists $j \in \{ 4, 6, 8 \}$ such that $x_j \in \{ x_2, x_3 \}$
and $\{x_1, x_2, x_3, x_4, x_5, x_6, x_7, x_8 \} \backslash \{ x_j \}$ are distinct.

By considering all the subcases stemming from Cases 1 and 2, we can bound $\mathbb{E}(|\mathcal{T}_n(\omega)|)$ by
$$
\mathbb{E}(|\mathcal{T}_n(\omega)|) \ll S_1 + S_2 + S_3 + S_4 + S_5 + S_6,
$$
where
$$
S_1 = \sum_{ \substack{  \deg x_i > M \ (i=1,..., 8)  \\ x_1 + x_2 + x_3 =n \\ x_1 + x_4 = x_5 + x_6 = x_7 + x_8} }
q^{ -\gamma (\deg x_1 + \deg  x_2 + \deg x_3 + \deg x_4 + \deg  x_5 + \deg x_6 + \deg x_7 + \deg x_8)  },
$$
$$
S_2 = \sum_{ \substack{ \deg x_i > M \ (i=1,2,3,5,6,7,8)
\\ x_1 + x_2 + x_3 =n \\ 2 x_1 = x_5 + x_6 =  x_7 + x_8 } }
q^{ -\gamma (\deg x_1 + \deg  x_2 + \deg x_3 + \deg  x_5 + \deg x_6 + \deg x_7 + \deg x_8)  },
$$
$$
S_3 = \sum_{  \substack{ \deg x_i > M \ (i=1,2,3,4,5,7,8)    \\ x_1 + x_2 + x_3 =n  \\ x_1 + x_4 = 2 x_5 = x_7 + x_8 } }
q^{ -\gamma (\deg x_1 + \deg  x_2 + \deg x_3 + \deg x_4 + \deg  x_5 + \deg x_7 + \deg x_8)  },
$$
$$
S_4 = \sum_{  \substack{ \deg x_i > M \ (i=1,2,3,4,5,6,7)    \\ x_1 + x_2 + x_3 =n \\ x_1 + x_4 = x_5 + x_6 = 2 x_7 } }
q^{ -\gamma (\deg x_1 + \deg  x_2 + \deg x_3 + \deg x_4 + \deg  x_5 + \deg x_6 + \deg x_7)  },
$$
$$
S_5 = \sum_{ \substack{ \deg x_i > M \ (i=1,2,3,5,6,7,8) 
\\ x_1 + x_2 + x_3 =n   \\ x_1 + x_2 =  x_5 + x_6 = x_7 + x_8 } }
q^{ -\gamma (\deg x_1 + \deg  x_2 + \deg x_3 + \deg  x_5 + \deg x_6 + \deg x_7 + \deg x_8)  },
$$
and
$$
S_6 = \sum_{ \substack{ \deg x_i > M \ (i=1,2,3,4,5,7,8) 
\\ x_1 + x_2 + x_3 =n   \\ x_1 + x_4 =  x_5 + x_2 = x_7 + x_8 } }
q^{ -\gamma (\deg x_1 + \deg  x_2 + \deg x_3 + \deg  x_4 + \deg x_5 + \deg x_7 + \deg x_8)  }.
$$
We note that $S_1$ corresponds to the subcase when
$\{x_1, x_2, x_3, x_4, x_5, x_6, x_7, x_8 \}$ are distinct,
and $S_2, S_3, S_4$ to the remaining subcases stemming from Case 1.
There are essentially two distinct subcases amongst the remaining subcases
stemming from Case 2. This is because we could relabel $x_2$ and $x_3$ as each other
without affecting the analysis, since they play the same role, and similarly with $x_6$ and $x_8$.
Thus, without loss of generality, it suffices to consider the situations
$x_4 = x_2$ and $x_6 = x_2$, which
correspond to $S_5$ and $S_6$, respectively. 

We can then show that $S_i \ll q^{(-1/11) \max \{ \deg n, M \} }$ for each $1 \leq i \leq 6$ (we can in fact obtain
smaller bounds for $1 < i \leq 6$), from which it follows that
$$
\mathbb{E}(|\mathcal{T}_n(\omega)|)
\ll q^{(-1/11) \max \{ \deg n, M \} }.
$$
In an effort to keep the paper concise, we only give details for the bounds of $S_2$ and $S_6$.
We note that the bound for $S_6$ requires the most calculation of all six.
The bounds for $S_1, S_3, S_4$ and $S_5$ can be achieved in a similar manner as for $S_2$ and $S_6$.

Bound for $S_2$: Since the characteristic of $\mathbb{F}_q$ is not $2$, we have $  \deg  (2 x_1) =  \deg x_1$.
By a repeated application of Lemma \ref{basic 1}, we have
\begin{eqnarray}
\label{S_2 first}
&&
\\
S_2
\notag
&\ll&
\sum_{ \deg x_1 > M } q^{- \gamma \deg x_1} \sum_{\stackrel{ x_2, x_3 \in \mathbb{F}_q[t] \backslash \{ 0 \}}{x_2 + x_3 = n - x_1}}q^{- \gamma (\deg x_2 + \deg x_ 3)}
\left( \sum_{\stackrel{ x_5, x_6 \in \mathbb{F}_q[t] \backslash \{ 0 \} }{x_5 + x_6 = 2 x_1} } q^{- \gamma (\deg x_5 + \deg x_ 6)} \right)^2
\\
&\ll&
\sum_{ \deg x_1 > M } q^{- \gamma \deg x_1 }
q^{  (1 - 2 \gamma) \max\{ \deg (n - x_1), 0 \}  } q^{( 2 - 4 \gamma) \deg (2 x_1)}
\notag
\\
&\ll&
q^{(2 - 5 \gamma ) \max\{ \deg n, M \}} +
\sum_{ \deg x_1 > M } q^{ (2 - 5 \gamma) \deg x_1}
q^{(1 - 2 \gamma) \deg (n - x_1)}.
\notag
\end{eqnarray}
We consider the two cases $\deg n \leq M$ and $\deg n > M$, separately.
Suppose $\deg n \leq M$. Since $7 \gamma - 3 = 16/11 > 1$, we have by
~(\ref{sum bound 1}) that
$$
S_2 \ll q^{(2 - 5 \gamma ) \max\{ \deg n, M \} } + \sum_{ \deg x_1 > M } q^{ (3 - 7 \gamma) \deg x_1}
\ll q^{ (4 - 7 \gamma) M}.
$$

On the other hand, suppose $\deg n > M$. We split and simplify the final sum in ~(\ref{S_2 first}) as follows,
\begin{eqnarray}
\label{S_2}
&&
\sum_{ M < \deg x_1 < \deg n } + \sum_{ \deg x_1 = \deg n }
+ \sum_{ \deg x_1 > \deg n } q^{ (2 - 5 \gamma) \deg x_1}
q^{(1 - 2 \gamma) \deg (n - x_1)}
\\
&\ll&
q^{(1 - 2 \gamma) \deg n } \sum_{ M < \deg x_1 < \deg n } q^{ (2 - 5 \gamma) \deg x_1}
+ q^{ (2 - 5 \gamma) \deg n} \sum_{ \deg x_1 = \deg n }
q^{(1 - 2 \gamma) \deg (n - x_1)}
\notag
\\
&&
\phantom{0000000011111111111111122222222233333333333333333}
+ \sum_{ \deg x_1 > \deg n }
q^{(3 - 7 \gamma) \deg  x_1}
\notag
\\
&\ll&
q^{(1 - 2 \gamma) \deg n } \sum_{ \deg x_1 < \deg n } q^{ (2 - 5 \gamma) \deg x_1}
+ q^{ (2 - 5 \gamma) \deg n} \sum_{ \deg z < 1 + \deg n }
q^{(1 - 2 \gamma) \deg z}
\notag
\\
&&
\phantom{00000000111111111111111222222222333333333333333333}
+ \sum_{ \deg x_1 > \deg n }
q^{(3 - 7 \gamma) \deg  x_1}.
\notag
\end{eqnarray}
Since $2 - 5 \gamma = -13/11 \leq -1$, by a similar calculation
as in ~(\ref{sum bound 2}) we see that
$$
\sum_{ \deg x_1 < \deg n } q^{(2 - 5 \gamma) \deg  x_1} \ll 1.
$$
By applying this estimate to the first sum, ~(\ref{sum bound 2}) to the second sum, and ~(\ref{sum bound 1}) to the third sum in the final expression of ~(\ref{S_2}), we obtain that
$$
S_2 \ll q^{(2 - 5 \gamma) \max\{ \deg n, M \} }+ q^{(1 - 2 \gamma) \deg n } + q^{ (4 - 7 \gamma) \deg n}  \ll q^{ (1 - 2 \gamma) \deg n}
$$
when $\deg n > M$. Therefore, by combining both cases together we obtain that
$$
S_2 \ll q^{(1 - 2 \gamma) \max\{ \deg n, M \} } \ll  q^{ - \frac{1}{11} \max\{ \deg n, M \} }.
$$

Bound for $S_6$: In order to simplify the sum, we make the following substitutions
$x_2 = n - x_3 - x_1$ and $x_5 = 2 x_1 + x_3 + x_4 - n$.
By Lemma \ref{basic 1}, we have
\begin{eqnarray}
\label{baaaaaaaa2}
&&
\\
\notag
S_6
&\ll&
\sum_{x_4 \in \mathbb{F}_q[t] \backslash \{ 0 \}, \deg x_1, \deg x_3 > M } 
q^{- \gamma (\deg x_1 + \deg (n - x_3 - x_1) )}
q^{- \gamma (\deg x_3 + \deg x_4)}
\\
&& \phantom{111112222222333333333333444444}
\cdot
q^{- \gamma \deg (2 x_1 + x_3 + x_4 - n)}
\sum_{\stackrel{ x_7, x_8 \in \mathbb{F}_q[t] \backslash \{ 0 \} } 
{x_7 + x_8 = x_1 + x_4 }} q^{- \gamma (\deg x_7 + \deg x_ 8)}
\notag
\\
&\ll&
\sum_{ x_4 \in \mathbb{F}_q[t] \backslash \{ 0 \} , \deg x_1, \deg x_3 > M } 
q^{- \gamma (\deg x_1 + \deg (n - x_3 - x_1) )}
q^{- \gamma (\deg x_3 + \deg x_4)}
\notag
\\
&& \phantom{11111222222233333333333344444444444}
\cdot
q^{- \gamma \deg (2 x_1 + x_3 + x_4 - n)} q^{( 1 - 2 \gamma) \max\{ \deg (x_1 + x_ 4) , 0 \} }
\notag
\\
&\ll&
\sum_{ \deg x_1, \deg x_3 > M}  
q^{- \gamma (\deg x_1 + \deg (n - x_3 - x_1) )} q^{- \gamma \deg x_3 }
\notag
\\
&& \phantom{1111122222223333333}
\cdot
\sum_{ x_4 \in \mathbb{F}_q[t] \backslash \{ 0 \}}
q^{- \gamma \deg x_4} q^{- \gamma \deg (2 x_1 + x_3 + x_4 - n)} q^{( 1 - 2 \gamma)  \max\{ \deg (x_1 + x_ 4), 0 \}  } .
\notag
\end{eqnarray}

For simplicity, let
$a = 2 x_1 + x_3 - n $. If $a \not =0$, then we simplify the inner sum in the final expression above by Lemma \ref{basic 2},
\begin{eqnarray}
\label{baaaaa}
&&\sum_{ x_4 \in \mathbb{F}_q[t] \backslash \{ 0 \} } 
q^{- \gamma \deg x_4} q^{- \gamma \deg (a + x_4)} q^{( 1 - 2 \gamma) \max\{ \deg (x_1 + x_ 4), 0 \} }
\\
&\ll&
q^{- \gamma \deg x_1} q^{- \gamma \deg (a - x_1)}
+ \sum_{ x_4 \in \mathbb{F}_q[t] \backslash \{ 0 \} } 
q^{- \gamma \deg x_4} q^{- \gamma \deg (a + x_4)} q^{( 1 - 2 \gamma)  \deg (x_1 + x_ 4)  }
\notag
\\
&\ll&
q^{- \gamma \deg x_1} q^{- \gamma \deg (a - x_1)}  +  q^{(1 - 2 \gamma) (\deg a + \deg x_1 )}
\notag
\\
&\ll&
q^{- \gamma \deg x_1}  +  q^{(1 - 2 \gamma) (\deg a + \deg x_1 )}.
\notag
\end{eqnarray}
If $a=0$, then we have
by Lemma \ref{basic 2} that
\begin{eqnarray}
\label{baaaaaaaaaa4}
&&\sum_{ x_4 \in \mathbb{F}_q[t] \backslash \{ 0 \} } 
q^{- \gamma \deg x_4} q^{- \gamma \deg (a + x_4)} q^{( 1 - 2 \gamma) \max\{ \deg (x_1 + x_ 4), 0 \} }
\\
&\ll&
q^{- 2 \gamma \deg x_1}
+ \sum_{ x_4 \in \mathbb{F}_q[t] \backslash \{ 0 \} } 
q^{- \gamma \deg x_4} q^{- \gamma \deg (a + x_4)} q^{( 1 - 2 \gamma)  \deg (x_1 + x_ 4)  }
\notag
\\
&\ll&
q^{- 2 \gamma \deg x_1}
+ q^{ (1 - 2 \gamma) \deg x_1 }
\notag
\\
&\ll&
q^{ (1 - 2 \gamma) \deg x_1 }.
\notag
\end{eqnarray}
By the bounds obtained in ~(\ref{baaaaa}) and ~(\ref{baaaaaaaaaa4}), the change of variable $x_3 = n - x_1 - z$, and Lemma \ref{basic 1}, we obtain from ~(\ref{baaaaaaaa2}) that
\begin{eqnarray}
\label{baa5}
&&
\\
\notag
S_6 &\ll&
\sum_{ x_3 \in \mathbb{F}_q[t] \backslash \{ 0 \}, \deg x_1>M}  
q^{- 2 \gamma \deg x_1}  q^{ - \gamma \deg (n - x_3 - x_1)} q^{- \gamma \deg x_3 }
\\
&& \phantom{11223} + \sum_{x_3 \in \mathbb{F}_q[t] \backslash \{ 0 \}, \deg x_1 > M } q^{- \gamma (\deg x_1 + \deg (n - x_3 - x_1) )} q^{- \gamma \deg x_3 } q^{(1 - 2 \gamma) (\deg a + \deg x_1 )}
\notag
\\
&& \phantom{11223344556677} +
\sum_{ \substack{ \deg x_1, \deg x_3 > M  \\ a = 2 x_1 + x_3 - n  = 0 } }
q^{- \gamma (\deg x_1 + \deg (n - x_3 - x_1) )} q^{- \gamma \deg x_3 } q^{ (1 - 2 \gamma) \deg x_1 }
\notag
\\
&\ll&
\sum_{ \deg x_1>M}
q^{- 2 \gamma \deg x_1}  q^{ (1 - 2 \gamma) \max \{ \deg (n - x_1) , 0  \} }
\notag
\\
&& \phantom{11223} +
\sum_{x_3 \in \mathbb{F}_q[t] \backslash \{ 0 \}, \deg x_1 > M } q^{(1 - 3 \gamma) \deg x_1  - \gamma \deg x_3 - \gamma \deg (n - x_3 - x_1)   + (1 - 2 \gamma) \deg (2 x_1 + x_3 - n ) }
\notag
\\
&& \phantom{1122334455667788999} +
\sum_{ \substack{ \deg x_1, \deg x_3 > M \\   2 x_1 + x_3   = n  } }
q^{- 2 \gamma \deg x_1} q^{- \gamma \deg x_3 } q^{ (1 - 2 \gamma) \deg x_1 }
\notag
\\
&\ll&
q^{- 2 \gamma \max\{ \deg n, M \} }
+
\sum_{ \deg x_1>M}
q^{- 2 \gamma \deg x_1}  q^{ (1 - 2 \gamma) \deg (n - x_1)  }
\notag
\\
&& \phantom{11223} +
\sum_
{\deg x_1 > M } q^{(1 - 3 \gamma) \deg x_1}
\sum_{z \in \mathbb{F}_q[t]  } 
q^{ - \gamma \deg z   - \gamma \deg (n - x_1 - z)   + (1 - 2 \gamma) \deg (x_1 - z) }
\notag
\\
&& \phantom{11223344555667788999000111222} +
\sum_{ \substack{ \deg x_1, \deg x_3 > M \\  2 x_1 + x_3 = n  } }
q^{- \gamma \deg x_3 } q^{ (1 - 4 \gamma) \deg (2 x_1) }.
\notag
\end{eqnarray}
Since $2 \gamma > 9/11$, we bound the first sum in the final expression above by Lemma \ref{basic 1}
as follows,
\begin{eqnarray}
\label{baaaaaaaaa3}
&&
\\
\notag
&&
\sum_{ \deg x_1 > M } q^{- 2 \gamma \deg x_1}
q^{(1 - 2 \gamma) \deg (n - x_1) }
\ll
\sum_{ \deg x_1 > M } q^{- 9/11 \deg x_1}
q^{(1 - 2 \gamma) \deg (n - x_1) }
\ll
q^{-\frac{1}{11} \max\{ \deg n, M \} }.
\end{eqnarray}
Since $1 - 4 \gamma = -17/11 < -5/11$, the third sum can be bounded by Lemma \ref{basic 1} in the following manner,
$$
\sum_{ \substack{ \deg x_1, \deg x_3 > M \\  2 x_1 + x_3 = n  } }
q^{- \gamma \deg x_3 } q^{ (1 - 4 \gamma) \deg (2 x_1) }
\ll
\sum_{ \substack{ \deg x_1, \deg x_3 > M \\  2 x_1 + x_3 = n  } }
q^{- \frac{7}{11} \deg x_3 } q^{ - \frac{5}{11} \deg (2 x_1) }
\ll
q^{- \frac{1}{11} \max \{ \deg n, M \}}.
$$
Therefore, by applying Lemma \ref{basic 2} to the inner sum of the remaining second sum,
we see that the final expression of ~(\ref{baa5}) can be bounded by
\begin{eqnarray}
S_6 &\ll& q^{- \frac{1}{11} \max \{ \deg n, M \}} +
\sum_{ \deg x_1 > M, n = x_1}
q^{(1 - 3 \gamma) \deg x_1}
q^{(1  - 2 \gamma )  \deg ( - x_1 ) }
\notag
\\
&& \phantom{112223334445566777778} + \sum_{ \deg x_1 > M} 
q^{(1 - 3 \gamma) \deg x_1}
q^{(1  - 2 \gamma ) (\deg (x_1 - n) + \deg ( - x_1 ) )}
\notag
\\
&\ll&
 q^{- \frac{1}{11} \max \{ \deg n, M \}} + q^{ (2 - 5 \gamma) \max \{ \deg n, M \}} +
\sum_{ \deg x_1 > M } 
q^{(2 - 5 \gamma) \deg x_1}
q^{(1  - 2 \gamma ) \deg (n - x_1)}.
\notag
\end{eqnarray}
Notice that the sum in the final estimate obtained above is the same as the sum in the estimate for $S_2$
in ~(\ref{S_2 first}). Therefore, we have by the work done to bound $S_2$ that
\begin{eqnarray}
S_6 \ll  q^{- \frac{1}{11} \max \{ \deg n, M \}} +  q^{(1 - 2 \gamma) \max \{ \deg n, M \} } \ll  q^{- \frac{1}{11} \max \{ \deg n, M \}}.
\notag
\end{eqnarray}

\end{proof}

\begin{replemma}{family bound}
We have the following bounds on the expectations.
\newline
\newline
i) $\mathbb{E}(|  \mathcal{U}_{r}(\omega)|) \ll q^{-\frac{3}{11} \max \{ \deg r, M \} }$.
\newline
ii) $\mathbb{E}(| \mathcal{V}_{r}(\omega)|) \ll q^{-\frac{3}{11} \max \{ \deg r, M \} }$.
\newline
iii) $\mathbb{E}(| \mathcal{W}_{r}(\omega)|) \ll q^{ - \frac{2}{11} \max \{ \deg r, M \} }$. 
\end{replemma}
\begin{proof}
By Lemma \ref{basic 1}, we obtain that
$$
\mathbb{E}(|  \mathcal{U}_{r}(\omega)|) \leq \sum_{\stackrel{\deg x,\deg y> M}{x+y=r}} q^{- \gamma \deg x} q^{- \gamma \deg y}
\ll q^{(1 - 2 \gamma) \max \{ \deg r, M \} } \ll q^{-\frac{3}{11} \max \{ \deg r, M \} },
$$
and
$$
\mathbb{E}(| \mathcal{V}_{r}(\omega)|) \leq \sum_{\stackrel{\deg x,\deg y> M}{x-y=r}} q^{- \gamma \deg x} q^{- \gamma \deg y}
\ll q^{(1 - 2 \gamma) \max \{ \deg r, M \} } \ll q^{-\frac{3}{11} \max \{ \deg r, M \} }.
$$
Similarly, we have by Lemma \ref{basic 1} that
\begin{eqnarray}
\label{W sum}
\mathbb{E}(| \mathcal{W}_{r}(\omega)|) &\leq& \sum_{ \stackrel{ \deg x_i > M \ (i = 4,5, 6, 7, 8), } 
{x_5 + x_6 = x_7 + x_8 =  r + x_4 }  } q^{- \gamma \deg x_4}q^{- \gamma \deg x_5}q^{- \gamma \deg x_6}q^{- \gamma \deg x_7}q^{- \gamma \deg x_8}
\\
&\leq&
\sum_{ \deg x_4 > M } q^{- \gamma \deg x_4}
\left(\sum_{\stackrel{\deg x_5, \deg x_6 > M} {x_5 + x_6 = r + x_4}  } q^{- \gamma \deg x_5}q^{- \gamma \deg x_6} \right)^2.
\notag
\\
&\ll&
\sum_{ \deg x_4 > M } q^{- \gamma \deg x_4}
q^{(2 - 4 \gamma) \max \{ \deg (r + x_4), M \} } .
\notag
\end{eqnarray}
With our choice of $\gamma = 7/11$, we have $0 < 4 \gamma - 2 < 1$ and
$5 \gamma -2 > 1$.
Therefore, by Lemma \ref{basic 3},
we obtain that
$$
\mathbb{E}(| \mathcal{W}_{r}(\omega)|) \ll q^{(3 - 5\gamma) \max\{ \deg r, M \} }.
$$
\end{proof}

\section{Estimates in Theorem \ref{main 2}
}
\label{calculations 2}
Recall in Section \ref{sec proof of main 2}, we work in the probability space $S_M(\gamma, S(\text{mod } n_0))$,
where $\gamma = \frac{2}{3} + \frac{\varepsilon}{9 + 9 \varepsilon}$ for some $\varepsilon > 0$ sufficiently small,
and $S$ is a non-empty subset of $\mathbb{G}_N$
satisfying the conditions of Corollary \ref{cor to thm}.
Since the computation in this section is similar to that of Section \ref{calculations 1},
we omit some of the details.

\begin{replemma}{E R_n}
We have that
$$
\mathbb{E}(| \mathcal{R}_n(\omega)|) \gg q^{ \frac{2 \varepsilon^2}{9 + 9 \varepsilon} \deg n},
$$
for $n \in \mathbb{F}_q[t] \backslash \{ 0 \}$ with $\deg n$ sufficiently large.
\end{replemma}
\begin{proof}
By the definition of $ \mathcal{R}_n(\omega)$, we have
$$
\mathbb{E}(| \mathcal{R}_n(\omega)|) = \sum_{ \{ x_1, x_2, x_3, x_4 \} \in \mathcal{R}_n } \mathbb{P}(x_1, x_2, x_3, x_4 \in \omega)
\geq q^{ - (3 + \varepsilon) \gamma \deg n} |R'_n|,
$$
where
\begin{eqnarray}
\mathcal{R}'_n = \{ \{ x_1, x_2, x_3, x_4 \} \in \mathcal{R}_n : x_i \equiv S \ (\text{mod } n_0),
\deg x_i = \deg n  > M \ (1 \leq i \leq 3),
\notag
\\
M < \deg x_4 \leq \varepsilon \deg n \}.
\notag
\end{eqnarray}
By our choice of $S$, we know there exist distinct $s_1, s_2, s_3, s_4$ such that
$n \equiv s_1 + s_2 + s_3 + s_4\ (\text{mod } n_0)$. We fix such  $s_1, s_2, s_3, s_4,$
and write $x_i = s_i + n_0 y_i$. Let $l \in \mathbb{F}_q[t]$ be the polynomial such that
$n - s_1 - s_2 - s_3 - s_4 = l n_0$.
Then we have $|\mathcal{R}'_n| \geq |\mathcal{R}^*_n|$, where
\begin{eqnarray}
\mathcal{R}^*_n = \{ \{ y_1, y_2, y_3, y_4 \} : y_1 + y_2 + y_3 + y_4 = l,
\deg y_i = \deg n - \deg n_0 \ (1 \leq i \leq 3),
\notag
\\
\deg y_4 \leq \varepsilon \deg n - \deg n_0 \}.
\notag
\end{eqnarray}
For each $y_4$ with $\deg y_4 \leq \varepsilon \deg n - \deg n_0$,
we can give a lower bound to the number of polynomials $y_1, y_2, y_3$ of degree $\deg n - \deg n_0$ such that $y_1 +  y_2 + y_3 = l - y_4$
in a similar manner as in Lemma \ref{E Q_n} or \cite[Lemma 6.7]{C}. Therefore, we obtain that
$$
|\mathcal{R}^*_n| \gg (q^2)^{\deg y_i} q^{ \varepsilon (\deg n - \deg n_0)} \gg q^{(2 + \varepsilon) \deg n}.
$$
Thus, we have our result
$$
\mathbb{E}(|\mathcal{R}_n(\omega)|) \geq q^{ - (3 + \varepsilon) \gamma \deg n} |\mathcal{R}'_n|
\gg q^{  \left( - \frac{\varepsilon}{3 + 3 \varepsilon} - \varepsilon \gamma + \varepsilon  \right) \deg n }
\gg q^{  \frac{2 \varepsilon^2 }{9 + 9 \varepsilon}  \deg n }.
$$
\end{proof}

\begin{repproposition}{Delta R_n}
We have that
$$
\Delta( \mathcal{R}_n) \ll q^{\frac{- 3 \varepsilon + 2 \varepsilon^2}{9 + 9 \varepsilon} \deg n},
$$
for $n \in \mathbb{F}_q[t] \backslash \{ 0 \}$ with $\deg n$ sufficiently large.
\end{repproposition}
\begin{proof}
Recall
$$
\Delta( \mathcal{R}_n) = \sum_{ \stackrel{\theta, \theta' \in \mathcal{R}_n}{\theta \sim \theta'}} \mathbb{P}(\theta, \theta' \subseteq \omega),
$$
where by $\theta \sim \theta'$ we mean $\theta \cap \theta' \not = \emptyset$
and $\theta \not = \theta'$. By the definition of $\mathcal{R}_n$, it is clear that if
$\theta \sim \theta'$, then $| \theta \cap \theta' | = 1$ or $2$.
We split $\Delta(\mathcal{R}_n)$ into several sums according to $\theta \cap \theta'$ in order to estimate it. 
We let $\theta = \{ x_1, x_2, x_3, x_4 \}$ and $\theta' = \{ x'_1, x'_2, x'_3, x'_4 \}$. 

Case 1.  $ \theta \cap \theta'  = \{ x_i \}$, where $\deg x_i \leq \varepsilon \deg n$.
Without loss of generality, let $i=1$.
Note $0 < \gamma < 1$ and $2 \gamma > 1$. We have by Lemmas \ref{basic 1} and \ref{basic 3} that
\begin{eqnarray}
L_1
&=&\sum_{\substack{ \deg x_i, \deg x'_j > M \ (1 \leq i \leq 4, 1\leq j \leq 3)
\\ x_1 + x_2 + x_3 + x_4 = n \\
x_1 + x'_2 + x'_3 + x'_4 = n \\ \deg x_1 \leq \varepsilon \deg n }}
q^{- \gamma (\deg x_1 + \deg x_2 + \deg x_3 + \deg x_4+ \deg x'_2 + \deg x'_3 + \deg x'_4)}
\notag
\\
\notag
\\
&\ll&
 q ^{( 4- 6 \gamma) \deg n} q^{(1- \gamma) \varepsilon \deg n }.
\notag
\end{eqnarray}

Case 2.  $ \theta \cap \theta'  = \{ x_j \}$,
where $\deg x_j > \varepsilon \deg n$.
Without loss of generality, let $j = 2$.
We have by Lemma \ref{basic 1} that
\begin{eqnarray}
\label{second delta case 2}
&&
\\
L_2 &=& \sum_{\substack{
\deg x_i, \deg x'_j > M \ (1 \leq i \leq 4, j=1,3,4) \\
\deg x_1, \deg x'_1 \leq \varepsilon \deg n  \\
x_3 + x_4 = n - x_1 - x_2\\ x'_3 + x'_4 = n - x'_1 - x_2 }}
q^{- \gamma (\deg x_1 + \deg x_2 + \deg x_3 + \deg x_4+ \deg x'_1 + \deg x'_3 + \deg x'_4)}
\notag
\\
&\ll&
q^{ ( (3 - 5 \gamma) + (2 - 2 \gamma ) \varepsilon ) \deg n }.
\notag
\end{eqnarray}

Case 3.  $ \theta \cap \theta'  = \{ x_i, x_j \}$, where $\deg x_i \leq \varepsilon \deg n$. Without loss of generality,
let $i = 1$ and $j = 2$.
Note we have $2 \gamma > 1$, $0 < 4 \gamma - 2 < 1$, and $5 \gamma -2 > 1$.
Thus, we obtain by Lemmas \ref{basic 1} and \ref{basic 3} that
\begin{eqnarray}
L_3 &=& \sum_{\substack{ \deg x_i, \deg x'_j > M \ (1 \leq i \leq 4, j = 3,4)
\\ x_1 + x_2 + x_3 + x_4 = n \\
x_1 + x_2 + x'_3 + x'_4 = n \\ \deg x_1 \leq \varepsilon \deg n }}
q^{- \gamma (\deg x_1 + \deg x_2 + \deg x_3 + \deg x_4 + \deg x'_3 + \deg x'_4)}
\notag
\\
&\ll&
q^{( 3 - 5 \gamma)  \deg n + ( 1 - \gamma)  \varepsilon \deg n }.
\notag
\end{eqnarray}

Case 4.  $ \theta \cap \theta'  = \{ x_j, x_k \}$, where $\deg x_j, \deg x_k > \varepsilon \deg n$. Without loss of generality,
let $j=2$ and $k=3$.
We obtain by Lemmas \ref{basic 1} and \ref{basic 3} that
\begin{eqnarray}
\label{second delta case 4}
&&
\\
\notag
L_4 &=&\sum_{\substack{ \deg x_i, \deg x'_j > M \ (1 \leq i \leq 4, j = 1,4)
\\ x_1 + x_2 + x_3 + x_4 = n \\
x'_1 + x_2 + x_3 + x'_4 = n \\ \deg x_1, \deg x'_1 \leq \varepsilon \deg n } }
q^{- \gamma (\deg x_1 + \deg x_2 + \deg x_3 + \deg x_4 + \deg x'_1 + \deg x'_4)}
\\
&\ll&
q^{ (2 - 2 \gamma) \varepsilon \deg n } q^{(1 - 2 \gamma) \deg n }.
\notag
\end{eqnarray}

Combining all the bounds computed for $L_1, L_2, L_3, L_4$, we obtain that
\begin{eqnarray}
\Delta( \mathcal{R}_n) &\ll&
L_1 + L_2 + L_3 + L_4
\notag
\\
&\ll&
q ^{( 4- 6 \gamma) \deg n} q^{(1- \gamma) \varepsilon \deg n }.
\notag
\end{eqnarray}
Note with our choice of $\gamma = \frac{2}{3} + \frac{\varepsilon}{9 + 9 \varepsilon}$, we have
$$
( 4 - 6 \gamma) + (1- \gamma) \varepsilon
=
- \frac{6 \varepsilon}{9 + 9 \varepsilon} + \left( \frac{1}{3} - \frac{ \varepsilon}{9 + 9 \varepsilon} \right) \varepsilon =
\frac{- 3 \varepsilon + 2 \varepsilon ^2 }{9 + 9 \varepsilon}.
$$
\end{proof}

\begin{replemma}{exp Bn}
We have that
$$
\mathbb{E}(| \mathcal{B}_n(\omega)|) \ll q^{ - \frac{ \varepsilon^2}{18} \max \{ \deg n, M \} }.
$$
\end{replemma}

\begin{proof}
Recall from Section \ref{sec proof of main 2} that
$$
\mathcal{B}_n = \{ \overline{x} = (x_1, x_2, x_3, x_4, x_5, x_6, x_7 ) : \overline{x} \text{ satisfies Cond} (\mathcal{B}_n)  \},
$$
where
\begin{eqnarray}
\text{Cond} (\mathcal{B}_n)
=
\begin{cases}
\{ x_1, x_2, x_3, x_4 \} \in \mathcal{R}_n \\
x_1 + x_5 = x_6 + x_7, \ \ \  \{x_1, x_5\} \not =  \{x_6, x_7 \} \\
x_1 \equiv x_6  \ (\text{mod } n_0), \ x_5 \equiv x_7 \ (\text{mod }n_0).
\end{cases}
\notag
\end{eqnarray}
If $\varepsilon \deg n \leq M$, then we have $\mathbb{E}(|\mathcal{B}_n(\omega)|) = 0$.
Therefore, we can bound $\mathbb{E}(| \mathcal{B}_n(\omega)|) \ll q^{- \frac{ \varepsilon^2}{18} M }$
if $\deg n \leq M < M / \varepsilon$, and $\mathbb{E}(|\mathcal{B}_n(\omega)|) \ll q^{- \frac{ \varepsilon^2}{18} \deg n}$
if $M < \deg n \leq M / \varepsilon$. Thus we suppose $\varepsilon \deg n > M$ for the remainder of the proof.

Let $(x_1, x_2, x_3, x_4, x_5, x_6, x_7) \in \mathcal{B}_n$. By arguing in a similar manner as in the
proof of Lemma \ref{exp Tn} or \cite[Lemma 6.8]{C}, we can verify that the following five situations are the only cases we need to consider.
\newline

Case 1: all the $x_i$'s are distinct.

Case 2: $x_6 = x_7$ and $\{ x_1, ..., x_6 \}$ are distinct.

Case 3: $x_5 = x_1$ and $\{ x_1, x_2, x_3, x_4, x_6, x_7\}$ are distinct.

Case 4: $x_5 \in \{ x_2, x_3, x_4\}$ and $\{ x_1, x_2, x_3, x_4, x_6, x_7\}$ are distinct.

Case 5: $x_7 \in \{ x_2, x_3, x_4\}$ and $\{ x_1, x_2, x_3, x_4, x_5, x_6\}$ are distinct.
\newline
\newline
By considering all of the five cases above, we obtain that $\mathbb{E}(|\mathcal{B}_n(\omega)|)$
can be bounded by
$$
\mathbb{E}(|\mathcal{B}_n(\omega)|) \ll S_1 + S_2 + S_3 + S_4 + S_5,
$$
where
$$
S_1 = \sum_{ \substack{ \deg x_i > M \ (1 \leq i \leq 7)
\\ x_1 + x_2 + x_3 + x_4 =n \\ x_1 + x_5 = x_6 + x_7
\\ \min_{1 \leq i \leq 4} \{ \deg x_i \} \leq \varepsilon \deg n  } }
q^{-\gamma (\deg x_1 + \deg x_2 +\deg x_3 +\deg x_4 +\deg x_5 +\deg x_6 +\deg x_7)},
$$
$$
S_2 = \sum_{ \substack{ \deg x_i > M \ (1 \leq i \leq 6)
\\ x_1 + x_2 + x_3 + x_4 =n \\ x_1 + x_5 = 2 x_6
\\ \min_{1 \leq i \leq 4} \{ \deg x_i \} \leq \varepsilon \deg n  } }
q^{-\gamma (\deg x_1 + \deg x_2 +\deg x_3 +\deg x_4 +\deg x_5 +\deg x_6 )},
$$
$$
S_3 = \sum_{ \substack{ \deg x_i > M \ (i = 1,2,3,4,6,7)
\\ x_1 + x_2 + x_3 + x_4 =n \\ 2 x_1 = x_6 + x_7
\\ \min_{1 \leq i \leq 4} \{ \deg x_i \} \leq \varepsilon \deg n  } }
q^{-\gamma (\deg x_1 + \deg x_2 +\deg x_3 +\deg x_4  +\deg x_6 +\deg x_7)},
$$
$$
S_4 = \sum_{ \substack{
\deg x_i > M \ (i=1,2,3,4,6,7)
\\ x_1 + x_2 + x_3 + x_4 =n \\ x_1 + x_2 = x_6 + x_7
\\ \min_{1 \leq i \leq 4} \{ \deg x_i \} \leq \varepsilon \deg n  } }
q^{-\gamma (\deg x_1 + \deg x_2 +\deg x_3 +\deg x_4  +\deg x_6 +\deg x_7)},
$$
and
$$
S_5 = \sum_{ \substack{ \deg x_i > M \ (1 \leq i \leq 6)
\\ x_1 + x_2 + x_3 + x_4 =n \\ x_1 + x_5 = x_6 + x_2
\\ \min_{1 \leq i \leq 4} \{ \deg x_i \} \leq \varepsilon \deg n  } }
q^{-\gamma (\deg x_1 + \deg x_2 +\deg x_3 +\deg x_4 +\deg x_5 +\deg x_6 )}.
$$
Note $S_1$, $S_2$, $S_3$, $S_4$, and $S_5$ correspond to Cases 1, 2, 3, 4, and 5, respectively.

By computing the bounds for each  $S_1, ..., S_5$ in a similar manner as in Lemma \ref{exp Tn} and \cite[Lemma 6.8]{C}, 
we obtain
\begin{eqnarray}
\mathbb{E}(|\mathcal{B}_n(\omega)|) &\ll& S_1 + S_2+ S_3 + S_4 + S_5
\notag
\\
&\ll& q^{- \frac{\varepsilon^2}{18} \deg n }
\notag
\end{eqnarray}
when $\varepsilon \deg n > M.$
\end{proof}

\small
\begin{thebibliography}{24}

\bibitem{C} J. Cilleruelo, \textit{On Sidon sets and asymptotic bases}, to appear in the Proceedings of the London Mathematical Society.

\bibitem{C1} J. Cilleruelo, \textit{Combinatorial problems in finite fields and Sidon sets},
Combinatorica, \textbf{32} (2012), 497 - 511.

\bibitem{DP} J.-M. Deshouillers and A. Plagne,
\textit{A {S}idon basis}, Acta Math. Hungar. \textbf{123}(2009), 233-238.

\bibitem{E1} P. Erd\H{o}s , \textit{The probability method: Successes and limitations},
J. Statist. Plann. Inference,  \textbf{72} (1998), 207 - 213.

\bibitem{E2} P. Erd\H{o}s , A. Sarkozy and V.T. S\'{o}s,
\textit{On additive properties of general sequences},
Discrete Math. \textbf{136} (1994), 75 - 99.

\bibitem{E3} P. Erd\H{o}s , A. Sarkozy and V.T. S\'{o}s,
\textit{On sum sets of Sidon sets I}, J. Number Theory \textbf{47} (1994), 329 - 347.

\bibitem{HR} H. Halberstam and K. F. Roth, \textit{Sequences}, Springer-Verlag, New York, 1983.

\bibitem{J} S. Janson, \textit{Poisson approximation for large deviations}, Random Structures Algorithms
\textbf{1} (1990), 221-229.

\bibitem{JLR} S. Janson, T. {\L}uczak and A. Rucinski,
\textit{Random graphs}, Wiley-Interscience, New York, 2000.

\bibitem{K1} S. Kiss,
\textit{On Sidon sets which are asymptotic basis}, Acta Math. Hungar. \textbf{128}(2010), 46-58.

\bibitem{KRS} S. Kiss, E. Rozgonyi and C. S\'{a}ndor,
\textit{On Sidon sets which are asymptotic basis of order four}, arXiv:1304.5749.

\bibitem{LW} S. Lang and A. Weil,
\textit{Number of points of varieties in finite fields}, Amer. J. Math.\textbf{76} (1954), 819-827.


\end{thebibliography}
\end{document}